\documentclass[reqno,dvipsnames]{amsart}
\usepackage{geometry}
\pdfoutput=1

\usepackage{subcaption}
\usepackage{amssymb, amsmath}
\usepackage{mathrsfs}
\usepackage{amscd} 
\usepackage{amsmath}
 \usepackage[normalem]{ulem}

\usepackage[english]{babel}
\usepackage{lmodern}      
\usepackage{float} 
 \usepackage{hyperref}
\hypersetup{
	colorlinks   = true,          
	urlcolor     = blue,          
	linkcolor    = blue,       
	citecolor   = red            
}

\usepackage{enumitem}
\setenumerate[1]{label=(\roman*)} 

\usepackage[normalem]{ulem}
\usepackage{graphicx}
\usepackage[curve]{xypic}

\usepackage{tikz}
\usetikzlibrary{arrows,decorations,shapes,shadows}
\usetikzlibrary{calc,intersections}
\usepackage{verbatim}

\newbox\mybox
\def\overtag#1#2#3{\setbox\mybox\hbox{$#1$}\hbox to
  0pt{\vbox to 0pt{\vglue-#3\vglue-\ht\mybox\hbox to \wd\mybox
      {\hss$\ss#2$\hss}\vss}\hss}\box\mybox}
\def\undertag#1#2#3{\setbox\mybox\hbox{$#1$}\hbox to 0pt{\vbox to
    0pt{\vglue#3\vglue\ht\mybox\hbox to \wd\mybox
      {\hss$\ss#2$\hss}\vss}\hss}\box\mybox}
\def\lefttag#1#2#3{\hbox to 0pt{\vbox to 0pt{\vglue -6pt\hbox to
      0pt{\hss$\ss#2$\hskip#3}\vss}}#1}
\def\righttag#1#2#3{\hbox to 0pt{\vbox to 0pt{\vglue -6pt\hbox to
      0pt{\hskip#3$\ss#2$\hss}\vss}}#1}
\let\ss\scriptstyle
\def\splicediag#1#2{\xymatrix@R=#1pt@C=#2pt@M=0pt@W=0pt@H=0pt}
\def\Dot{\lower.2pc\hbox to 2pt{\hss$\bullet$\hss}}
\def\Circ{\lower.2pc\hbox to 2pt{\hss$\circ$\hss}}
\def\Vdots{\raise5pt\hbox{$\vdots$}}
\newcommand\lineto{\ar@{-}}
\newcommand\dashto{\ar@{--}}
\newcommand\dotto{\ar@{.}}

\DeclareMathOperator{\length}{length}

\newtheorem*{theorem*}{Theorem}
\newtheorem{theorem}{Theorem}[section]
\newtheorem{proposition}[theorem]{Proposition}

\newtheorem{thm}{Theorem}[section]

\newtheorem{prop}[thm]{Proposition}
\newtheorem{lemm}[thm]{Lemma}
\newtheorem{coro}[thm]{Corollary}
\newtheorem{rema}[thm]{Remark}

\newtheorem{defi}[thm]{Definition}

\newtheorem{exam}[thm]{Example}

\usepackage{xcolor}

\newcommand{\red}[1]{{\color{red} #1}}

\title{Inner rates of finite morphisms}
\author{Yenni Cherik}
\address{Aix-Marseille Université, CNRS, Centrale Marseille, I2M, Marseille, France}
\email{\href{mailto:yenni.cherik@univ-amu.fr}{yenni.cherik@univ-amu.fr}}

\begin{document}



\maketitle

\begin{abstract}
Let $(X, 0)$ be a complex analytic surface germ embedded in $(\mathbb{C}^n,0)$ with an isolated singularity and $\Phi=(g,f):(X,0) \longrightarrow (\mathbb{C}^2,0)$ be a finite morphism. We define a family of analytic invariants of the morphism $\Phi$, called \textit{inner rates} of $\Phi$. By means of the inner rates we study the polar curve associated to the morphism $\Phi$ when fixing the topological data of the curve $(gf)^{-1}(0)$ and the surface germ $(X,0)$, allowing to address a problem called polar exploration. We also use the inner rates to study the geometry of the Milnor fibers of a non constant holomorphic function $f:(X,0) \longrightarrow (\mathbb{C},0)$.
The main result is a formula which involves the inner rates and the polar curve alongside topological invariants of the surface germ $(X,0)$ and the curve $(gf)^{-1}(0)$.  
\end{abstract}

\section*{Introduction}

Let $(X, 0)$ be a complex analytic surface germ with an isolated singularity and $\Phi = (g, f) : (X, 0) \longrightarrow (\mathbb{C}^2, 0)$ be a finite morphism. The polar curve of the morphism $\Phi$ is the curve $\Pi_{\Phi}$ defined as the topological closure of the ramification locus of $\Phi$. Polar curves play an important role in the study of the geometry and the topology of germs of singular complex varieties, see, e.g.\cite{Teissier1982,EGBT,PH,BNP,BFP}.

In this paper we introduce and study a family of analytic invariants of the morphism $\Phi$ that we call \textit{inner rates} of $\Phi$, generalizing the notion of inner rates of a complex analytic surface germ $(X,0)$ first introduced by Birbrair, Neumann and Pichon in \cite{BNP}  as metric invariants, and later defined in \cite{BFP} by Belotto, Fantini and Pichon. Our main result (Theorem \ref{thmA}) establishes a formula, the \textit{inner-rates formula}, which relates the inner rates with analytical data of the polar curve $\Pi_{\Phi}$  and, in particular, provides a concrete way to compute the inner rates. It  has an equivalent version in terms of the laplacian of a certain graph (Corollary \ref{laplace}) which is a broad generalization of the \textit{Laplacian formula} \cite[Theorem 4.3]{BFP}.
An important motivation of our result concerns the study of Milnor fibers. Consider a non-constant holomorphic function $f: (X,0) \to (\mathbb{C},0)$ and a generic linear form (Definition \ref{genericnashrelative}) $\ell : (X,0) \to (\mathbb{C},0)$. We provide an interpretation of the inner rates of the morphism $\Phi = (\ell,f):(X,0) \to (\mathbb{C}^2,0)$ in terms of metric properties of the Milnor fiber, Theorem \ref{thmB}.

As an application of our methods, we study the problem of \emph{polar exploration} associated to $\Phi$, following \cite{BFP,BFNP}. Roughly speaking, it is the study of the relative position on $(X,0)$ of the two curves  $\Pi_{\Phi}$ and $(gf)^{-1}(0)$ i.e., of the relative positions of their strict transform by a resolution of $(X,0)$. More precisely it is the problem of determining the embedded topological type (Definition \ref{embeddedtopologicaltype}) of the union $\Pi_{\Phi} \cup (gf)^{-1}(0)$ from that of the curve $(gf)^{-1}(0)$. This is a natural problem which has been studied by many authors such as  Merle, Garc\'{\i}a Barroso, Delgado, Maugendre, Kuo, Parusi\'nski, Michel, Belotto, Fantini, Némethi, Pichon,  among others (see, e.g.,\cite{merle,KTP,EGB,DM2003,Michel2008,MaugendreMichel2017,DM2021,BFNP,BFP2}). An important contribution in the general case was made by Michel in \cite{Michel2008} via the \emph{Hironaka quotients} of the morphism $\Phi$ (Definition \ref{defihironaka}). New techniques and results involving inner rates were recently developed in \cite{BFNP} in the case where $\Phi$ is a generic linear projection of $(X,0)$, and a complete answer to polar exploration was given in \cite{BFP2} in the case where  $(X,0)$ is a Lipschitz  normally embedded surface germ. In the present paper, we address the case of a general finite morphism $\Phi$. We show the relation between the inner rates and the Hironaka quotients (Theorem \ref{thmD}$(ii)$), and we give an alternative proof of \cite[Theorem 4.9]{Michel2008} based on our inner rates formula. Finally, we give a family of examples where the inner rates formula provides tighter restrictions  on the relative position of the polar curve than the ones obtained from previous methods (Proposition \ref{propositionA}).\\

We now present our results in a sharp form. Let $\pi:(X_{\pi},E) \longrightarrow (X,0)$ be a \textit{good resolution of singularities} of $(X,0)$, that is, a proper bimeromorphic map which is an isomorphism on the complementary of a  simple normal crossing divisor $\pi^{-1}(0)=E,$ called the \textit{exceptional divisor}. Let $E_v$ be an irreducible component of $E$. We call \textit{curvette} of $E_v$ a smooth curve germ which intersects transversely $E_v$ at a smooth point of $E$. 

Throughout the paper,  we use the big-Theta asymptotic notation of Bachmann--Landau in the following form: given two function germs $h_1,h_2\colon \big([0,\infty),0\big)\to \big([0,\infty),0\big)$ we say that $h_1$ is a  \emph{big-Theta} of $h_2$ and we write   $h_1(t) = \Theta \big(h_2(t)\big)$ if there exists real numbers $\eta>0$ and $K >0$ such that for all $t \in [0,\eta)$ we have ${K^{-1}}h_2(t) \leq h_1(t) \leq K h_2(t)$.

Let $(u_1,u_2)=(g,f)$ be the  coordinates of $\mathbb{C}^2$. We prove  (Proposition \ref{inner-rate}) the existence of a rational number $q_{g,v}^f$ such that  for any smooth point $p$ of $E$ in $E_v$ which does not meet the strict transforms of $f^{-1}(0), g^{-1}(0)$ or the polar curve $\Pi_{\Phi}$  and for any pair of  disjoint curvettes $\gamma_1^*$ and $\gamma_2^*$ of $E_v$  passing through points of $E$ close enough to $p$,  $$ \mathrm{d}(\gamma_1 \cap \{u_2 = \epsilon \}, \gamma_2 \cap \{u_2= \epsilon \} ) =\Theta(\epsilon^{q_{g,v}^f}),$$where $\gamma_1=\Phi \circ \pi (\gamma_1^*)$, $\gamma_2=\Phi \circ \pi (\gamma_2^*)$ and $\mathrm{d}$ is the standard hermitian metric of $\mathbb{C}^2$. We call the number $q_{g,v}^f$ \textit{inner rate of $f$ with respect to $g$ along $E_v$}.

As already mentioned, the notion of inner rates associated to a finite morphism generalizes the inner rates of  \cite{BFP}. Indeed, suppose that $(X,0)$ is embedded in $\mathbb{C}^n$. Assume that $\pi$ is a good resolution which factors through the blowup of the maximal ideal and the \textit{Nash transform} (see  e.g \cite[Introduction]{Spivakovsky1990} for the definition of the Nash transform) and  that $\Phi$ is a "generic" linear projection in the sense of \cite[Subsection 2.2]{BFP}.  Then, the inner rates associated to $\Phi$ coincide with the inner rates of the surface germ $(X,0)$ (see \cite[Definition 3.3]{BFP}), which are metric invariants of  $(X,0)$. \\
 


Let us now state the main result of this paper, which is proved in section \ref{section3}. Let $\Gamma_{\pi}$ be the \textit{dual graph} of the good resolution $\pi$, that is, the graph whose vertices are in bijection with the irreducible  components of $E$ and  such that the  edges between  the vertices $v$ and $v'$  corresponding to $E_v$ and $E_{v'}$ are in bijection with   $E_v \cap E_{v'}$. Each vertex $v$ of this graph is weighted with the self intersection number $E_v^2$ and the genus $g_v$ of the corresponding complex curve $E_v$. Let $\mathrm{val}_{\Gamma_{\pi}}(v):=\left(\sum_{i \in V(\Gamma_{\pi}),  i \neq v }E_i  \right)\cdot E_v$ be the \textbf{valency} of $v$. We denote by $V(\Gamma_{\pi})$ the set of vertices of $\Gamma_{\pi}$ and $E(\Gamma_{\pi})$ the set of edges. Let us denote by $m_v(f)$  the order of vanishing of the function $f\circ \pi$ along the irreducible component $E_v$ of $E$ and by $f^*$ the strict transform of the curve $f^{-1}(0)$ by $\pi$.

\begin{theorem}[{The inner rates formula}] \label{thmA} Let $(X,0)$ be a complex surface germ with an isolated singularity and let $\pi :(X_{\pi},E) \longrightarrow (X,0) $ be a good resolution of $(X,0)$. Let $g,f:(X,0) \longrightarrow (\mathbb{C},0)$ be two holomorphic functions on $X$ such that the morphism $\Phi=(g,f): (X,0) \longrightarrow (\mathbb{C}^2,0)$ is finite. Let $E_{v_1},E_{v_2},\dots,E_{v_n}$ be the irreducible components of $E$. Let $M_{\pi}=(E_{v_i} \cdot E_{v_j})_{i,j \in \{1,2,\ldots,n\}}$ be the intersection matrix of the dual graph $\Gamma_{\pi}$. Consider the four following vectors: $a_{g,\pi}^f:=(m_{v_1}{(f)}q_{g, v_1}^f,\ldots,m_{v_n}(f)q_{g, v_n}^f)$, $K_{\pi} :=( \mathrm{val}_{\Gamma_{\pi}} (v_1) +2g_{v_1}-2,\ldots,\mathrm{val}_{\Gamma_{\pi}} (v_n) +2g_{v_n}-2)$, the \textbf{$F$-vector} $F_{\pi}=(f^* \cdot E_{v_1},\ldots,f^* \cdot E_{v_n} )$  and  the \textbf{$\mathcal{P}$-vector} $P_{\pi}=(\Pi_{\Phi}^* \cdot E_{v_1},\ldots,\Pi_{\Phi}^* \cdot E_{v_n})$ . Then we have:
$$M_{\pi}  .\underline{a_{g,\pi}^f}=\underline{K_{\pi}}+\underline{F_{\pi}}-\underline{P_{\pi}}.$$
Equivalently, for each irreducible component $E_v$ of $E$ we have the  following:
$$
 \left( \sum_{i \in V(\Gamma_{\pi})} m_{i}(f)q_{g,i}^f E_i  \right) \cdot E_{v}=\mathrm{val}_{\Gamma_{\pi}}(v) +f^* \cdot E_v-\Pi_{\Phi}^* \cdot E_v+2g_v-2,
$$
 where "$\cdot$" denotes the intersection number between curves.\end{theorem}

\noindent The idea of the proof is to relate the inner rates to the canonical divisor of the complex surface $X_{\pi}$ and then apply the classical adjunction formula on the irreducible components of the exceptional divisor $E=\pi^{-1}(0)$. This proof is quite shorter and different from the one provided in \cite{BFP} which relies on topological tools.

 Not only does theorem \ref{thmA} gives us a concrete way to compute the inner rates in terms of a  good resolution $\pi$ of $(X,0)$ just by computing the polar curve (see Example \ref{patricio}), but, since the intersection matrix $M_{\pi}$ is negative definite by a result of Mumford \cite[§1]{mum},  it also proves that given the dual graph $\Gamma_{\pi}$ together with the $F$-vector, the inner rates $q_{g,v}^f$  determines and are determined by the  $\mathcal{P}$-vector $(\Pi_{\Phi}^*\cdot E_1, \dots, \Pi_{\Phi}^* \cdot E_n)$. This fact allows us to  study polar curves by means of the inner rates, more specifically, to address the problem of \textit{polar exploration}, which we now describe. Suppose that $\pi:(X_\pi,E) \longrightarrow (X,0)$ is the minimal good resolution of $(X,0)$ and $(gf)^{-1}(0)$, that is, the minimal good resolution of $(X,0)$ such that $f^* \cup g^*$ is a disjoint union of curvettes of $E$. Following \cite{BFP}, polar exploration consists in answering the following question: is it possible to determine the $\mathcal{P}$-vector from the data of the  dual graph $\Gamma_{\pi}$, the $F$-vector and the $G$-vector?
 
An important contribution to this problem was made by Michel (\cite[Theorem 4.9]{Michel2008}) by means of the \textit{Hironaka quotients}. The Hironaka quotient of the morphism $\Phi$ associated to an irreducible component $E_v$ of $E$ is the rational number  $h_{g,v}^f=\frac{m_v(g)}{m_v(f)}, v \in V(\Gamma_{\pi})$. In this paper we will improve this result of Michel  by using the inner rates instead of the Hironaka quotients. The reason why the inner rates are more efficient than the Hironaka quotients for polar exploration comes from the following result.

\begin{theorem}\label{thmD}
Let $(X,0)$ be a complex surface germ and let  $\Phi=(g,f):(X,0) \longrightarrow (\mathbb{C}^2,0)$ be a finite morphism. Let $\pi:(X_{\pi},E)\longrightarrow (X,0)$ be a good resolution of $(X,0)$. There exists a subgraph $\mathcal{A}_{\pi}$ of $\Gamma_{\pi}$ such that: 

\begin{enumerate}
\item Call $f$-node any vertex $w$ of $\Gamma_{\pi}$ such that $f^*\cdot E_w \neq 0$. Let $v$ be   a vertex of $\Gamma_{\pi} $. There exists a path  from an $f$-node to $v$  in $\Gamma_{\pi}$  along which the inner rates are strictly increasing. 

\item The inner rates  and the Hironaka quotients of $\Phi$  coincide on $\mathcal{A}_{\pi}$.

\item The Hironaka quotients are constant on  $\Gamma_{\pi} \backslash \mathcal{A}_{\pi}$.
\end{enumerate}
\end{theorem}
 Points $(i)$ and $(ii)$  are proved in the sections \ref{section6} and \ref{section7}, while the point $(iii)$ comes from \cite[Theorem 1]{MaugendreMichel2017}.


 In Section \ref{section8}, we  show that {\cite[Theorem 4.9]{Michel2008}}  can be obtained as a consequence of Theorems \ref{thmA} and \ref{thmD}. We state it here taking account of  \cite[Theorem 1]{MaugendreMichel2017}:

\begin{theorem*}[{\cite[Theorem 4.9]{Michel2008}}]  Let $(X,0)$ be a complex surface germ with an isolated singularity and let $g,f:(X,0) \longrightarrow (\mathbb{C},0)$ be two holomorphic functions on $X$ such that the morphism $\Phi=(g,f): (X,0) \longrightarrow (\mathbb{C}^2,0)$ is finite. Let $\pi :(X_{\pi},E) \longrightarrow (X,0) $ be a good resolution of $(X,0)$. Let $\mathcal{A}_{\pi}$ be a subgraph of $\Gamma_{\pi}$ as in the statement of Theorem \ref{thmD}. Let $\mathcal{Z}$ be  a connected component of $\overline{\Gamma_{\pi} \backslash \mathcal{A}_{\pi}}$ or a single vertex on the complementary of  $\overline{\Gamma_{\pi} \backslash \mathcal{A}_{\pi}}$, then :


$$  \sum_{v \in \mathcal{Z} } m_v(f) \Pi_\Phi^* \cdot   E_v=- \left(  \sum_{v \in \mathcal{Z}} m_v(f) \chi'_v  \right).$$ where $\chi'_v:=2-2g_v-\mathrm{val}(v) -f^* \cdot E_v-g^* \cdot E_v$.

\end{theorem*}

It is now natural to ask: can we get a better restriction for the value of the $\mathcal{P}$-vector when using the inner rates  and their properties? We give a positive answer to this question via the following result (see proposition \ref{famille} for a more precise statement).

\begin{proposition}\label{propositionA}

There exists a sequence of graphs with arrows  $(\Gamma_n)_{n\geq 2}$ such that for each $n$:
\begin{enumerate}
\item There exists a  complex surface singularity  $(X_n,0)$ and a finite morphism $\Phi_n=(g_n,f_n) :( X_n,0) \longrightarrow (\mathbb{C}^2,0)$ such that   $\Gamma_n$ is the dual graph of the minimal good resolution of $(X_n,0)$ and $(g_nf_n)^{-1}(0)$;

\item The $\mathcal{P}$-vector of any such morphism $\Phi_n$ belongs to a set of $n+5$ elements.
\end{enumerate}

  \end{proposition}
  
   If one  performs a polar exploration on this family of examples using only \cite[Theorem 4.9]{Michel2008}, one obtains an exponential bound for the number of $\mathcal{P}$-vectors (see Remark \ref{remaramanujan}), while Proposition \ref{propositionA} provides a linear bound.


Another  important motivation to study the inner rates is to provide analytic invariants of the Milnor fibration. Let $(X,0)$ be a complex analytic surface germ embedded in $\mathbb{C}^n$ and let $f:(X,0) \longrightarrow (\mathbb{C},0)$ be a non constant holomorphic function. Let $\pi:(X_{\pi},E) \longrightarrow (X,0)$ be a good resolution which factors through the blowup of the maximal ideal and the Nash transform relative to $f$ (Definition \ref{nashgauss}). Let $\ell$ be a \textit{generic linear form} with respect to $\pi$ (Definition \ref{genericnashrelative}). The inner rates $q_{\ell, v}^f, v \in V(\Gamma_{\pi})$ associated to the morphism $(\ell,f)$ do not depend  of the choice  of the generic linear form $\ell$, we then denote them $q_v^f, v \in V(\Gamma_{\pi})$. In this case the inner rates $q_v^f$ gives informations on the \textit{inner metric}  of the Milnor fibers:

\begin{theorem}[{Theorem \ref{Milnor}}] \label{thmB} Let $(X,0) \subset (\mathbb{C}^n,0)$ be a complex surface germ with an isolated singularity at the origin of $\mathbb{C}^n$ and  let $f:(X,0) \longrightarrow (\mathbb{C},0)$  be a non constant holomorphic function. Let $ \pi: (X_{\pi},E) \longrightarrow (X,0)$ be a good resolution which factors through the Nash transform of $X$ relative to $f$ and the blowup of the maximal ideal.

Let   $\gamma_1^*$ and $\gamma_2^*$  be two curvettes of an irreducible component $E_v$ of the exceptional divisor $E$ such that $\gamma_i^*\cap f^*=\emptyset$ for $i \in \{1,2\}$.
Then there exists $q_v^f \in \mathbb{Q}_{>0}$ such that 
$$ \mathrm{d}_{\epsilon}(\gamma_1 \cap f^{-1}(\epsilon),\gamma_2 \cap f^{-1}(\epsilon)) = \Theta(\epsilon^{q_v^f})$$where $\gamma_1=\pi(\gamma_1^*)$,$\gamma_2=\pi(\gamma_2^*)$ and $\mathrm{d}_{\epsilon}$ is the Riemanian metric induced by $\mathbb{C}^n$ on the Milnor fiber $f^{-1}(\epsilon)$. Furthermore we have $q_v^f=q_{\ell,v}^f$ whenever $\ell$ is a generic linear form with respect to $f$ and $\pi$.
\end{theorem}\noindent Theorem \ref{thmB} is a relative version (with respect to $f$) and a generalization of \cite[Lemma 3.2]{BFP}. Indeed, one obtains \cite[Lemma 3.2]{BFP} by taking $f$ a generic linear form in the sense of \cite[Subsection 2.2]{BFP}.

As an application of Theorems \ref{thmA} and \ref{thmB} we give a generalization of a  result of Garc\'{\i}a Barroso and Teissier on the asymptotic behavior of the integral of the Lipschitz-Killing curvature (Definition \ref{langevincourbure}) along Milnor fibers. Let $\pi:(X_{\pi},E) \longrightarrow (X,0)$ be a good resolution which factors through the blowup of the maximal ideal and the Nash transform relative to $f$. Let $E_v$ be an irreducible component of $E$ and let  $\mathcal{N}(E_v,\epsilon), \epsilon >0$ be a family of tubular neighborhoods of $E_v$ in $X_{\pi}$ such that $\lim\limits_{\epsilon \rightarrow 0}\mathcal{N}(E_v,\epsilon)=E_v $ and such that  $\mathrm{Horn}(\epsilon,v):= \pi(\mathcal{N}(E_v,\epsilon) )$ is included in  $B_{\epsilon}$. Consider the set $F_{\epsilon,t}^v=f^{-1}(t) \cap \mathrm{Horn}(\epsilon,v)$ and let  $\delta_{\epsilon}^{\ell}>0$   be such that for any complex number $t$ with $|t| \leq  \delta_{\epsilon}^{\ell}$, we have $ \mathrm{Card}\{ F_{\epsilon,t}^v \cap \Pi_{\ell} \} = m_v(f) \Pi_\ell^* \cdot E_v$ where  $\ell$ is a generic linear form with respect to $\pi$. Set $\delta_{\epsilon}=\mathrm{min}_{\ell}\{\delta_{\epsilon}^{\ell} \}$.

\begin{theorem}[{Theorem \ref{concentration}}] \label{thmE} Let $(X,0)$ be a complex surface germ with an isolated singularity embedded in $(\mathbb{C}^n,0)$ and let $f:(X,0) \longrightarrow (\mathbb{C},0)$ be a non constant holomorphic function germ. Let $\pi:(X_{\pi},E) \longrightarrow (X,0)$ be a good resolution of $(X,0)$ which factors through the Nash transform of $X$ relative to $f$ and through the blowup of the maximal ideal of $(X,0)$. Let $v$ be a vertex of $\Gamma_\pi$, then:

$$\lim\limits_{\epsilon \rightarrow 0, |t| <{\delta_{\epsilon}}}  \int_{p \in F_{\epsilon,t}^v} K_{F_{\epsilon,t}^v}(p) \mathrm{d}V=\frac{\pi \omega_{2} }{2\omega_{2n-1}} \mathrm{Vol}(\mathbb{G}^{n-1}(\mathbb{C}^n)) \mathrm{C}_f , $$
where   $$C_f=m_v(f)\left( 2g_v-2+ \mathrm{Val}_{\Gamma_{\pi}}(v)+f^* \cdot E_v-\sum_{i \in V(\Gamma_{\pi}) }m_i(f)q_{i}^fE_i \cdot E_v \right).$$ and $\omega_i$ is the volume of the unit sphere $\mathbb{S}^{i}$.
\end{theorem}
Theorem \ref{thmE} generalizes the work of Garc\'{\i}a Barosso and Teissier in \cite{GarciaBarrosoTeissier1999} which treat the case of germs of holomorphic functions at the origin of $\mathbb{C}^2$.

\subsection*{Acknowledgments} 
I would like to express my deep gratitude to my thesis advisors André Belotto and Anne Pichon for their help and enthusiastic encouragements during the  preparation of this paper. I would also like to thank Patricio Almirón for fruitful 
 conversations about polar curves and in particular for pointing out Example \ref{patricio}. This work has been supported by the \textit{Centre National de la Recherche Scientifique (CNRS)} which funds  my PhD  scholarship.

\tableofcontents

 \section{Resolution of curves and surfaces}\label{section1}
In this section we introduce some classical materials   as they are presented in the introductions of \cite{MaugendreMichel2017},\cite{Michel2008}  and \cite{BFP}.

\begin{defi}
Let $(X,0)$ be a complex surface germ with an isolated singularity. A \textbf{resolution} of $(X,0)$ is a proper bimeromorphic map $\pi : (X_{\pi},E) \longrightarrow (X,0)$ such that $X_{\pi}$ is a smooth  complex surface  and the restricted  function $\pi_{|X_{\pi} \backslash E}: X_{\pi} \backslash E  \longrightarrow X \backslash 0$  is a biholomorphism. The curve $E=\pi^{-1}(0)$ is called the \textbf{exceptional divisor}. \\ The resolution $\pi : (X_{\pi},E) \longrightarrow (X,0)$ is   \textbf{good}  if $E$ is a simple normal crossing divisor i.e., it has smooth compact irreducible components and the singular points of $E$ are  transversal double points. Let $E_v$ be an irreducible component of $E$. A \textbf{curvette}  of $E_v$ is a  smooth curve germ which intersects  $E_v$ transversely at a smooth point of $E$.   \end{defi}
\begin{defi}
Let $\pi : (X_{\pi},E) \longrightarrow (X,0)$ be a  resolution of $(X,0)$ and let $(C,0)$ be a curve germ  in $(X,0)$. The \textbf{strict transform} of $C$ by $\pi$ is the curve  $C^*$  in $X_{\pi}$ defined as the topological closure of the set $\pi^{-1}(C \backslash \{0\})$. Let $E_1,E_2,\ldots,E_n$ be the irreducible components of $E$ and $h:(X,0) \longrightarrow (\mathbb{C},0) $ be a holomorphic function.  The \textbf{total transform} of $h$ by $\pi$  is the principal divisor $(h \circ \pi)$ on $X_\pi$, i.e., 
$$ (h \circ \pi)= \sum_{i=1}^n m_i(h) E_i+h^*$$
where $m_i(h)$ is the order of vanishing  of the holomorphic function $h  \circ \pi $ on the irreducible component $E_i$ of $E$ and $h^*$ is the strict transform of the curve $h^{-1}(0)$.
\end{defi}\begin{prop}[{\cite[Theorem 2.6]{laufer1972} or \cite[2.6]{nem} for a topological proof}]\label{laufer}
Let $\pi : (X_{\pi},E) \longrightarrow (X,0)$ be a  resolution of $(X,0)$. Let $h:(X,0) \longrightarrow (\mathbb{C},0) $ be a holomorphic function, then we have the following: \begin{equation*} (h \circ \pi) \cdot E_v=0  , \ \ \forall v \in V(\Gamma_{\pi}) .\end{equation*} where "$\cdot$" denote the intersection multiplicity between curves.
\end{prop}

\begin{defi}
Let $(C,0)$ be a curve germ in $(X,0)$.  A proper bimeromorphic  map $\pi : (X_{\pi},E) \longrightarrow (X,0)$ is  a \textbf{good resolution of $(X,0)$ and $(C,0)$} if it is a  good resolution of  $(X,0)$ such that
the strict transform $C^*$ is a disjoint union of curvettes.
\end{defi}
\begin{defi}
The \textbf{dual graph} of  a good resolution $\pi :(X_\pi,E) \longrightarrow (X,0) $ of $(X,0)$ is the graph $\Gamma_\pi$ whose vertices are in bijection with the irreducible  components of $E$ and  such that the  edges between  the vertices $v$ and $v'$  corresponding to $E_v$ and $E_{v'}$ are in bijection with   $E_v \cap E_{v'}$, each vertex $v$ of this graph is weighted with the self intersection number $E_v^2$ and the genus $g_v$ of the corresponding curve $E_v$. We denote by $V(\Gamma_{\pi})$ the set of vertices of $\Gamma_{\pi}$ and $E(\Gamma_{\pi})$ the set of edges. The \textbf{valency of a vertex $v$} is the number  $\mathrm{val}_{\Gamma_{\pi}}(v):=\left(\sum_{i \in V(\Gamma_{\pi}),  i \neq v }E_i  \right)\cdot E_v$. Let $\Phi=(g,f):(X,0) \longrightarrow (\mathbb{C},0)$ be a finite morphism  and let $E_v \subset E$ be an irreducible component  which meets the strict transform $g^*$ (resp. $f^*$), following  notations of \cite{MaugendreMichel2017}, we attach to the vertex $v$ corresponding to $E_v$ a going-out arrow (resp. a going-in arrow) weighted with the intersection number $g^* \cdot E_v$ (resp. $f^* \cdot E_v$).\end{defi}\begin{exam}\label{exemple1}
Consider the finite morphism $\Phi=(g,f):(\mathbb{C}^2,0) \longrightarrow (\mathbb{C}^2,0)$, where $g(x,y)=x+y$ and $f(x,y)=y^5-x^{12}$. The minimal good resolution $\pi:(X_{\pi},E) \longrightarrow (\mathbb{C}^2,0)$  of the curve $(gf)^{-1}(0)$ has the following dual graph: all the irreducible components of $E$  have genus $0$.

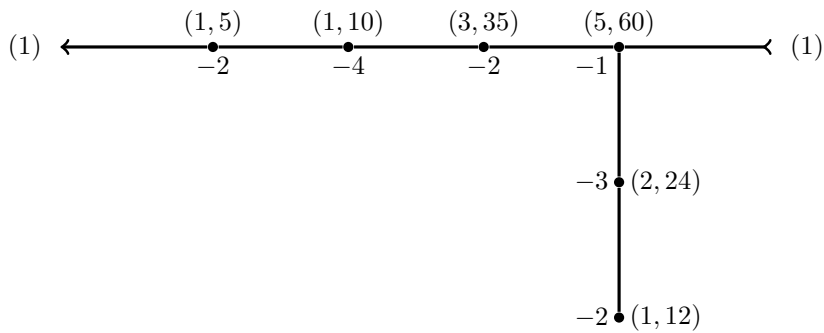
\begin{figure}[H]
\begin{tikzpicture}[node distance=4.5cm, very thick]
 \tikzstyle{titleVertex}      = [ shape=circle,node distance=4cm] \tikzstyle{inVertex}      = [ shape=circle,node distance=2.5cm]
\tikzstyle{Vertex}      = [fill, shape=circle, line cap=round,line join=round,>=triangle 45,scale=.4,font=\scriptsize]
  \tikzstyle{Edge}        = [black]
    \tikzstyle{arrowEdge}        = [black, ->]  
   \tikzstyle{arrowEdge'}        = [black, -<]    
     \tikzstyle{arrowEdge''}        = [red, ->]      
 
 \node[Vertex]      (1)               {};

     \node[Vertex]      (2)  [right       of=1] {};
  \node[Vertex]      (3)  [ right of=2]  {};
  \node[Vertex]      (4)  [right  of=3]  {};
  \node[Vertex]      (5)  [below        of=4] {};
  \node[Vertex]      (6)  [below       of=5] {};
    \node[inVertex] (0)  [ left     of=1]  {$(1)$};    
      \node[inVertex] (7)  [ right     of=4]  {$(1)$};   

          \path 
(1) edge[Edge] (2)
(2) edge[Edge] (3)
(3) edge[Edge] (4)
(4) edge[Edge] (5)
(4) edge[arrowEdge']  (7)
(1) edge[arrowEdge]  (0)
(5) edge[Edge] (6);

\draw (1) node[below] {$-2$}
          (2) node[below] {$-4$}
          (3) node [below]{$-2$}
           (4) node [below left ]{$-1$}
            (5) node [right]{$$}
             (6) node [right]{$$}
              (1) node[above] {$(1,5)$}
          (2) node[above] {$(1,10)$}
          (3) node [above]{$(3,35)$}
           (4) node [above]{$(5,60)$}
            (5) node [right]{$(2,24)$}
             (6) node [right]{$(1,12)$}
                 (5) node [left]{$-3$}
             (6) node [left]{$-2$};

                                      
  \end{tikzpicture} 
  \caption{The numbers between parenthesis are the orders of vanishing $(m_v(g),m_v(f))$ of  the functions $g \circ \pi $ and $f \circ \pi$ along the irreducible components of $E$, these numbers can be determined from the dual graph using Proposition \ref{laufer}.}
   \end{figure}




  \end{exam}

  \section{Inner rates of a finite morphism}\label{section2}

Let $(X,0)$  be a complex surface  germ with an isolated singularity and $g,f :(X,0) \longrightarrow (\mathbb{C},0)$  two holomorphic functions such that the morphism $\Phi=(g,f) :(X,0) \longrightarrow (\mathbb{C}^2,0)$ is finite. The aim of this section is to define  the notion of \textbf{inner rates}  associated to the morphism $\Phi=(g,f)$. This definition is a generalization of the  \textbf{inner rates of a complex surface germ with an isolated singularity} first introduced in \cite{BNP} and later in \cite[Definition 3.3]{BFP}.



\begin{defi} [{See e.g \cite[Definition 0.1]{KTP}}]  The \textbf{polar curve} of the morphism $\Phi $ is the curve $\Pi_{\Phi}$  defined as the topological closure of the critical locus of the finite morphism $\Phi=(g,f)$, that is
$$ \Pi_{\Phi} = \overline{  \{   x \in  X \backslash \{0\} \ | \ d_x \Phi : T_x X \longrightarrow \mathbb{C}^2 \ \text{ is non surjective} \}   }$$
\end{defi}

\begin{defi}(See e.g \cite[section 3]{BFP})
Given two functions germs $h_1,h_2 :([0,+\infty),0) \longrightarrow([0,+\infty),0)$, we say that $h_1$ is a \textbf{big-Theta} of $h_2$ and we write $h_1(t)=\Theta(h_2(t))$ if there exists  real numbers $\eta > 0$ and $K>0$ such that for all $t \in [0,\eta)$, 
$$ \frac{1}{K} h_2(t) \leq h_1(t) \leq K h_2(t)$$
\end{defi}


\begin{prop}\label{inner-rate}
Let $\pi :(X_{\pi},E) \longrightarrow (X,0)$ be a good resolution of $(X,0)$ and let $E_v$ be an irreducible component of the exceptional divisor $E$. Let us denote by $(u_1,u_2)=(g,f)$ the coordinates of $\mathbb{C}^2$ and by $\mathrm{d}$ the standard hermitian metric of $\mathbb{C}^2$. 
 Then there exists a positive rational number $q_{g,v}^f$ such that  for every smooth point $p$ of $E$  in $E_v \backslash (f^* \cup g^* \cup \Pi_{\Phi}^*)$, there exists an open neighborhood  $O_p \subset E_v$ of $p$ such that for every pair of curvettes $\gamma_{1}^*, \gamma_{2}^*$ of $E_v$ verifying:
 
\setcounter{equation}{0}
\renewcommand{\theequation}{2.\arabic{equation}}\begin{equation}\label{cond}
\left\{
\begin{array}{rcl}
 \gamma_1^* \cap \gamma_2^*&=&\emptyset \\
\gamma_{i}^* \cap O_p &\neq& \emptyset \ \text{for} \ i=1,2.
\end{array}
\right.
\end{equation}
we have: $$ \mathrm{d}(\gamma_1 \cap \{u_2 = \epsilon \}, \gamma_2 \cap \{u_2= \epsilon \} ) =\Theta(\epsilon^{q_{g,v}^f}),$$ where  $\gamma_1=(\Phi \circ \pi)(\gamma_1^*), \gamma_2=(\Phi \circ \pi) (\gamma_2^*)$ and $\epsilon \in \mathbb{R}$. Furthermore the number $m_v(f)q_{g,v}^f$ is an integer.
\end{prop}
\begin{defi}\label{inner}
We call  $q_{g,v}^f$ the \textbf{inner rate of $f$ with respect to $g$ along $E_v$}.
\end{defi}
Let us explain why the definition \ref{inner} is a generalization of the notion of inner rates of a complex surface germ with an isolated singularity as defined in \cite[Definition 3.3]{BFP}. Let $(X,0)$ be a complex surface germ embedded in  $\mathbb{C}^n$ with an isolated singularity and let $\pi:(X_\pi,E) \longrightarrow (X,0)$ be a good resolution which factors through the blowup of the maximal ideal and the \textit{Nash transform} (see  e.g \cite[Introduction]{Spivakovsky1990} for the definition of the Nash transform). Let $\ell_1$ and $\ell_2$ be two linear forms of $\mathbb{C}^n$ such that the projection $\ell=(\ell_1,\ell_2)$ is a generic projection in the sense of \cite[Subsection 2.2]{BFP}. For every irreducible component $E_v$ of $E$, by \cite[Lemma 3.2]{BFP}, the inner rate $q_{\ell_1, v}^{\ell_2}$ is equal to the inner rate of the complex surface germ $(X,0)$ associated with the component $E_v$.

\begin{proof}[Proof of Proposition \ref{inner-rate}]
Let $p$ be a smooth  point of $E_v$ which does not belong to  the strict transforms $f^*,g^*$ and $\Pi_{\Phi}^*$. Let $(x,y)$ be a local system of coordinates of $X_{\pi}$ centered at $p$ such that $E_v$ has  local  equation $x=0$ and such that $(f \circ \pi)(x,y)=x^{m_v(f)}$ .
Let  $U$ be the  unit of $\mathbb{C}\{ x,y\}$ such that $ (g \circ \pi)(x,y) = x^{m_v(g)}U(x,y)$. We can write 
$$ U(x,y)=\sum_{i \geq 0} a_{i0}x^i  + \sum_{j \geq 1}y^j \sum_{i \geq 0} a_{ij} x^i$$
Since $\Phi$ is a  finite morphism  the set $\{ i \geq 0 \ | \ \exists j > 0, a_{ij} \neq 0  \}$ is non empty. Let $k$ be its minimal element. Then,
$$ U(x,y)=g_{0}(x)  +x^k \sum_{j \geq 1}y^j g_{j}(x),$$
where  $g_{0}(x)=\sum_{i \geq 0} a_{i0}x^i$ and $g_{j}(x) =\sum_{i \geq 0} a_{ij} x^{i-k}$. Note that the set $\{ j > 0 \ | \ g_{j}(0) \neq 0  \}$ is non empty.

Setting $q_p:=\frac{m_v(g)+k}{m_v(g)}$, we then have:
\begin{eqnarray}\label{forme}
(g \circ \pi)(x,y) = x^{m_v(g)}g_{0}(x)+x^{q_pm_v(g)}\sum_{j \geq 1} y^j g_{j}(x).   \end{eqnarray}

\noindent Let $\gamma_{1}^*$ and $\gamma_{2}^*$ be two curvettes of $E_v$ parametrized respectively by
$$ t \mapsto (t,\alpha + th_1(t)), \ t \mapsto (t,\beta + th_2(t) ), \ \alpha,\beta \in \mathbb{C} $$
where $h_1$ and $h_2$ are convergent power series. The curves $\gamma_1= (\Phi \circ \pi)(\gamma_1^*) $ and $\gamma_2=(\Phi \circ \pi)(\gamma_2^*)$  are then parametrized respectively by
$$  t \mapsto (t^{m_v(g)}g_{0}(t)+t^{q_pm_v(g)}\sum_{j \geq1}(\alpha+
th_1(t))^j g_{j}(t) ,t^{m_v(f)})$$ and
$$  t \mapsto (t^{m_v(g)}g_{0}(t)+t^{q_pm_v(g)}\sum_{j \geq1}(\beta+
th_2(t))^j g_{j}(t) ,t^{m_v(f)}).$$
Therefore, for $\epsilon > 0$, we have $$\mathrm{d}(\gamma_1 \cap \{ u_2 = \epsilon \}, \gamma_2 \cap \{ u_2 =\epsilon\}) 
=\displaystyle\left\lvert \epsilon^{\frac{q_p m_v(g)}{m_v(f)}} \displaystyle\right\rvert H(\epsilon), $$ where 
$$H(\epsilon)=\displaystyle\left\lvert\sum_{j \geq 1 }\left( (\alpha+ \epsilon^{\frac{1}{m_v(f)}}h_1(\epsilon^{\frac{1}{m_v(f)}}))^j - (\beta + \epsilon^{\frac{1}{m_v(f)}}h_2(\epsilon^{\frac{1}{m_v(f)}}))^j \right) g_{j}(\epsilon^{\frac{1}{m_v(f)} } ) \displaystyle\right\rvert.$$ We need to prove that  $$ H(0)=\displaystyle\left\lvert \sum_{j \geq 1 }\left( \alpha^j - \beta^j \right) g_{j}(0) \displaystyle\right\rvert.$$is non zero $0$ when $\alpha$ and $\beta$ are distinct and in a small enough neighborhood of the origin of $\mathbb{C}$.
Let $j_0 > 0$ be the minimal element of the set $\{ j > 0 \ | \ g_{j}(0) \neq 0  \}$ then 
$$ H(0)=\displaystyle\left\lvert (\alpha^{j_0}-\beta^{j_0})g_{j_0}(0) + \sum_{j \geq 1, j \neq j_0 }\left( \alpha^j - \beta^j \right) g_{j}(0) \displaystyle\right\rvert.$$ Now, let us prove  that $j_0=1$. Let us compute in the coordinates $(x,y)$ the equation of the total transform of the polar curve.
 The jacobian matrix of $ \Phi \circ \pi$ is  $$\mathrm{Jac}( \Phi \circ \pi)(x,y) = \begin{pmatrix}
* & x^{q_{p} m_v(g)}(g_{1}(x)+2yg_{2}(x)+\ldots) \\
m_v(f) x^{m_v(f)-1} &0
\end{pmatrix}$$then $\mathrm{Det}(\mathrm{Jac}( \Phi \circ \pi)(x,y))=m_v(f) x^{q_{p} m_v(g)+m_v(f)-1}(g_{1}(x)+2yg_{2}(x)+\ldots)=0$ is the equation of the total transform of $\Pi_{\Phi}$. Since $g_{j_0}(0) \neq 0$ it follows that the equation of the  strict transform of the polar curve $\Pi_{\Phi}^*$ is    $$\sum_{j \geq 1 }  jy^{j-1} g_{j}(x) =0.$$
By hypothesis $p \notin \Pi_{\Phi}^*$ which means that $g_{1}(0) \neq 0$ and then $j_0 =1$. Thus
$$ H(0)=|(\alpha-\beta)g_{1}(0) + \sum_{j > 1}\left( \alpha^j - \beta^j \right) g_{j}(0) |.$$  If $\alpha$ and $\beta$ are distinct and in a small enough neighborhood of $0 \in \mathbb{C}$ we have $H(0) \neq 0$. 
With this, we can conclude that there exists a neighborhood $O_p$ of $p$ such that for every curvettes satisfying \eqref{cond}
$$\mathrm{d}(\gamma_1 \cap \{ u_2 = \epsilon \}, \gamma_2 \cap \{ u_2 =\epsilon\}) =\displaystyle\left\lvert \epsilon^{ \frac{q_p m_v(g)}{m_v(f)} } \displaystyle\right\rvert H(\epsilon)= \Theta( \epsilon^{\frac{m_v(g)q_p}{m_v(f)}} ).$$ Now we will see that the number $\frac{q_pm_v(g)}{m_v(f)}$ does not depend of the point $p$. Let $p'$ be another smooth point of $E$ in $E_v$  such that $p' \notin f^* \cup g^* \cup \Pi_{\Phi}^*$. Let $c:[0,1] \longrightarrow E_v  $   be a continuous  injective path which does not pass through $f^*$, $g^*$ or  $\Pi_{\Phi}^*$ such that $c(0)=p$ and $c(1)=p'$. By compacity of $E_v,$ we can  choose  a finite family of open sets $\{ O_{p_i} , i \in \{1,\ldots,s\}\}$ which covers the curve $c[0,1]$ and  such that for every pair of curvettes  $\gamma_{i,1}^*$ and $\gamma_{i,2}^*$ of $E_v$ passing through two different points of $O_{p_i}$ there exists $a_i \in \mathbb{R}$ which verifies
$$ \mathrm{d}(\gamma_{i,1} \cap \{ u_2 = \epsilon \}, \gamma_{i,2} \cap \{ u_2 =\epsilon\})=\Theta(\epsilon^{a_i}),$$ where $\gamma_{i,1}=\Phi_{\pi}(\gamma_{i,1}^*)$ and $\gamma_{i,2}=\Phi_{\pi}(\gamma_{i,2}^*)$.
We can suppose, even if it means refining and reordering the open cover $\{ O_i\}$, that $a_1=\frac{q_pm_v(g)}{m_v(f)}    ,a_s=\frac{q_{p'}m_v(g)}{m_v(f)}$   and that $O_{p_i} \cap O_{p_{i+1}} \neq \emptyset $ for $i =1,2,\ldots,s-1$.\\
Let $\gamma_{i,i+1}^*$ and $\gamma_{i,i+1}'^*$ be two curvettes passing through two different points of $O_{p_i} \cap O_{p_{i+1}}$, then, by definition of the open cover $\{O_{p_i}\}_i$ we have
$$ \mathrm{d}(\gamma_{i,i+1} \cap \{ u_2 = \epsilon \},  \gamma_{i,i+1}' \cap \{ u_2 =\epsilon\})=\Theta(\epsilon^{a_i})=\Theta(\epsilon^{a_{i+1}}),$$ where  $\gamma_{i,i+1}= \Phi \circ \pi(\gamma_{i,i+1}^*),\gamma_{i,i+1}'= \Phi \circ \pi(\gamma_{i,i+1}'^*).$ Finally, we get that $\frac{q_pm_v(g)}{m_v(f)}=a_1=a_2=\ldots=a_s=\frac{q_{p'}m_v(g)}{m_v(f)}:=q_{g,v}^f$. Furthermore $m_v(f) q_{g,v}^f=m_v(f) \frac{q_pm_v(g)}{m_v(f)}$  is an integer.
 \end{proof}
 \noindent The following remark will be used in the proofs of the inner rates formula \ref{thmA} and of  Lemma \ref{recurence}.
 \begin{rema}\label{remarque}
By equation \eqref{forme},  for any smooth point $p$ of $E$ in $E_v$ with $p \notin f^* \cup g^* \cup \Pi_{\Phi}^*,$ there exists a local system of coordinates $(x,y)$ centered at $p$ such that
 
 $$\left \{
\begin{array}{rcl}
(u_1 \circ  \Phi \circ \pi)(x,y)&=&x^{m_v(g)} g_{0}(x) +x^{ q_{g,v}^{f} m_v(f) }\sum_{j \geq 1 }  y^j g_{j}(x) \\
(u_2 \circ  \Phi \circ \pi)(x,y)&=&x^{m_v(f)}
\end{array}
\right.
$$ 
with  $g_{0}(0) \neq 0,g_1(0)\neq 0$.
 \end{rema}

\section{Inner rates formula}\label{section3}
In this section we  state and prove the main theorem of this article the \textbf{inner rates formula} \ref{thmA}. 
\begin{thm}[{The inner rates formula}] \label{laplacien} Let $(X,0)$ be a complex surface germ with an isolated singularity and let $\pi :(X_{\pi},E) \longrightarrow (X,0) $ be a good resolution of $(X,0)$. Let $g,f:(X,0) \longrightarrow (\mathbb{C},0)$ be two holomorphic functions on $X$ such that the morphism $\Phi=(g,f): (X,0) \longrightarrow (\mathbb{C}^2,0)$ is finite. Let $M_{\pi}=(E_{v_i} \cdot E_{v_j})_{i,j \in \{1,2,\ldots,n\}}$ be the \textbf{intersection matrix} of the dual graph $\Gamma_{\pi}$, $a_{g,\pi}^f:=(m_{v_1}{(f)}q_{g, v_1}^f,\ldots,m_{v_n}(f)q_{g, v_n}^f)$, $K_{\pi} :=( \mathrm{val}_{\Gamma_{\pi}} (v_1) +2g_{v_1}-2,\ldots,\mathrm{val}_{\Gamma_{\pi}} (v_n) +2g_{v_n}-2)$, $F_{\pi}=(f^* \cdot E_{v_1},\ldots,f^* \cdot E_{v_n} )$ be \textbf{the $F$-vector } and 
$P_{\pi}=(\Pi_{\Phi}^* \cdot E_{v_1},\ldots,\Pi_{\Phi}^* \cdot E_{v_n})$ be \textbf{the polar vector or $\mathcal{P}$-vector}. Then we have:
$$M_{\pi}  .\underline{a_{g,\pi}^f}=\underline{K_{\pi}}+\underline{F_{\pi}}-\underline{P_{\pi}}.$$
Equivalently, for each irreducible component $E_v$ of $E$ we have the  following:
$$
 \left( \sum_{i \in V(\Gamma_{\pi})} m_{i}(f)q_{g,i}^f E_i  \right) \cdot E_{v}= \mathrm{val}_{\Gamma_{\pi}}(v)+f^* \cdot E_v-\Pi_{\Phi}^* \cdot E_v+2g_v-2,
$$
\end{thm}

\noindent As mentioned in the introduction, this result is a generalization of the "laplacian formula" proved by Belotto, Fantini and Pichon  \cite[Theorem 4.3]{BFP}. Their proof relies on topological tools. Here, we are going to prove the theorem \ref{laplacien} by using as main ingredient the classical  adjunction formula. That we state now. 
\begin{defi}[{See e.g \cite[Chapter 3. Subsection 6.3]{shafarevich}}]
Let $S$ be a smooth complex surface. The \textbf{canonical divisor} of $S$, denoted $K_S$, is the  divisor $ K_{\Omega}$ associated to any meromorphic $2$-form $\Omega$ defined on $S$. It is  well defined up to linear equivalence i.e.,  for any pair of $2$-forms $\Omega$ and $\Omega'$ of $S$ there exists a meromorphic function $h$ on $S$  such that $$ K_{\Omega}=(h)+K_{\Omega'}$$ where  $K_{\Omega}$ and $K_{\Omega'}$ are respectively the divisors associated to the $2$-forms $\Omega$ and $\Omega'$.
\end{defi}

\begin{thm}[{Adjunction formula, see e.g \cite[Chapter 4. Subsection 2.3]{shafarevich}}] \label{adj} Let $S$ be a complex surface and $C \subset S$ be  a compact Riemann surface embedded in $S$. Then 
$$ (K_S+C) \cdot C=2g_{C}-2, \ \ \  \text{where} \ g_{C}  \ \text{is the genus of } \ C.  $$

\end{thm}

\begin{proof}[Proof of Theorem \ref{laplacien}]
Let us set $\Phi_{\pi}= \Phi \circ \pi $. Let $E_v$ be an irreducible component of $E$. Let us denote by $(u_1,u_2)=(g,f)$ the coordinates of $\mathbb{C}^2$. By remark \ref{remarque}, given a smooth point $p$ of $E$ in $E_v$  with $p \notin f^* \cup g^* \cup \Pi_{\Phi}^* $ there exists a  local system  coordinates $(x,y)$ centered at  $p$ such that $$\left \{
\begin{array}{rcl}
(u_1 \circ \Phi_\pi)(x,y)&=&x^{m_v(g)} g_{0}(x) +x^{ q_{g,v}^{f} m_v(f) }\sum_{j \geq 1 }  y^j g_{j}(x) \\
(u_2 \circ \Phi_\pi)(x,y)&=&x^{m_v(f)}
\end{array}
\right.
$$ with  $g_{0}(0) \neq 0$ and $g_{1}(0) \neq 0$.
Let us consider the holomorphic $2$-form $\omega$ on $\mathbb{C}^2$  defined by $\omega=\mathrm{d}u_1 \wedge \mathrm{d}u_2$. Let $\Omega :=\Phi_{\pi}^*\omega$ be the pullback of $\omega$ by the holomorphic function $\Phi_{\pi}$. In the neighborhood of $p$, the $2$-form $\Omega$ is given in the coordinates $(x,y)$ by: $$
\Omega=-m_v(f)x^{q_{g,v}^{f}m_v(f)+m_v(f)-1}\sum_{j \geq 1 }  jy^{j-1} g_{j}(x) \mathrm{d}x \wedge \mathrm{d}y.
$$ Since $\Phi_{\pi}= \Phi \circ \pi $  is a local  isomorphism on the complement of   $E  \cup \Pi_{\Phi}^*$ in $X_{\pi}$, the   $2$-form $\Omega$ does not vanish  on this set. Therefore, the canonical divisor $K_{X_{\pi}}$ of the smooth complex surface $X_{\pi}$ is represented by 
$$ K_{\Omega}= \sum_{i \in V(\Gamma_{\pi})} (q_{g,i}^{f}m_i(f)+m_i(f)-1)E_i+\Pi_{\Phi}^*. $$
We now apply the ajunction formula \ref{adj} to the compact Riemann surface  $E_v$  \begin{eqnarray}\label{adjunction}
(K_{\Omega}+E_v) \cdot E_v=2g_v-2.\end{eqnarray} Replacing $K_{\Omega} $ by its value in the equation \eqref{adjunction}, we obtain:  $$ \sum_{i \in V(\Gamma_{\pi})} q_{g,i}^{f}m_i(f) E_i \cdot E_v+ \sum_{i \in V(\Gamma_{\pi} )} m_i(f)E_i\cdot E_v -\sum_{i \in V(\Gamma_{\pi}), i \neq v }E_i \cdot E_v +\Pi_{\Phi}^* \cdot E_v=2g_v-2$$Finally, by Proposition \ref{laufer} we have: $$\left( \sum_{i \in V(\Gamma_{\pi} )} m_i(f)E_i\cdot E_v \right)=\left(-f^* \cdot E_v\right).$$ Replacing this in the previous equation, we get the desired equality: $$ \sum_{i \in V(\Gamma_{\pi})} q_{g,i}^{f}m_i(f) E_i \cdot E_v= \mathrm{val}_{\Gamma_{\pi}}(v) +f^* \cdot E_v-\Pi_{\Phi}^* \cdot E_v+2g_v-2.$$\end{proof}\begin{rema}
Since the intersection matrix of the dual graph $\Gamma_{\pi}$ associated with a good resolution $\pi:(X_{\pi},E) \longrightarrow (X,0)$ is negative definite. Theorem \ref{laplacien} implies that given the dual graph $\Gamma_{\pi}$ together with the $F$-vector  $(f^*\cdot E_1, \dots, f^*\cdot E_n)$, the inner rates $q_{g,v}^f, v \in V(\Gamma_{\pi})$ determines and are determined by the  $\mathcal{P}$-vector $(\Pi_{\Phi}^*\cdot E_1, \dots, \Pi_{\Phi}^* \cdot E_n)$.
\end{rema}

\noindent Let us give an example where we compute the inner rates using the theorem \ref{laplacien}.



\begin{exam}\label{exemple3.5}
Consider the case where $(X,0)=(\mathbb{C}^2,0)$. Let $f(x,y)=y^5-x^{12}$ and $g(x,y)=x+y$  and consider the finite morphism $\Phi=(g,f): (\mathbb{C}^2,0) \longrightarrow (\mathbb{C}^2,0) $. The polar curve  $\Pi_{\Phi}$ of this morphism  has equation $(\frac{\partial f }{\partial y}.\frac{\partial g }{\partial x})(x,y) -(\frac{\partial f }{\partial x}.\frac{\partial g }{\partial y})(x,y) =5y^{4}+12x^{11}=0$. Consider the minimal good resolution $\pi :(X_{\pi},E) \longrightarrow (\mathbb{C}^2,0)$  of the curve $(gf)^{-1}(0)$. We attach each vertex $v$ of $\Gamma_{\pi}$ corresponding  to an irreducible component $E_v$ which meet the strict transform of the polar curve with a red going-out arrow weighted with the intersection number $\Pi_{\Phi}^* \cdot E_v$.\begin{figure}[H]
\begin{tikzpicture}[node distance=3.8cm, very thick]
 \tikzstyle{titleVertex}      = [ shape=circle,node distance=4cm]
  \tikzstyle{inVertex}      = [ shape=circle,node distance=2cm]
\tikzstyle{Vertex}      = [fill, shape=circle, line cap=round,line join=round,>=triangle 45,scale=.4,font=\scriptsize]
  \tikzstyle{Edge}        = [black]
    \tikzstyle{arrowEdge}        = [black, ->]  
    \tikzstyle{arrowEdge'}        = [black, -<]    
    \tikzstyle{arrowEdge''}        = [red, ->]

 \node[Vertex]      (1)               {};

     \node[Vertex]      (2)  [right       of=1] {};
  \node[Vertex]      (3)  [ right of=2]  {};
  \node[Vertex]      (4)  [right  of=3]  {};
  \node[Vertex]      (5)  [below        of=4] {};
  \node[Vertex]      (6)  [below       of=5] {};
     \node[inVertex] (0)  [left     of=1]  {$(1)$};
 
          \node[inVertex] (8)  [below left     of=5]  {$\red{(1)}$};    
         \node[inVertex] (9)  [below left     of=6]  {$\red{(2)}$};    
                       \node[inVertex] (7)  [ right     of=4]  {$(1)$};

          \path 
            (1) edge[Edge] (2)
      (2) edge[Edge] (3)
             (3) edge[Edge] (4)
               (4) edge[Edge] (5)
                (1) edge[arrowEdge]  (0)
                               (4) edge[arrowEdge']  (7)
                                   (5) edge[arrowEdge'']  (8)
                                           (6) edge[arrowEdge'']  (9)

                                                                      (5) edge[Edge] (6);
\draw (1) node[below] {$-2$}
(0) node[below right] {$$}
          (2) node[below] {$-4$}
          (3) node [below]{$-2$}
           (4) node [below left ]{$-1$}
            (5) node [right]{$$}
             (6) node [right]{$$}
              (1) node[above] {$v_1(5)$}
          (2) node[above] {$v_2(10)$}
          (3) node [above]{$v_5(35)$}
           (4) node [above]{$v_6(60)$}
            (5) node [right]{$v_4(24)$}
             (6) node [right]{$v_3(12)$}
                 (5) node [left]{$-3$}
             (6) node [left]{$-2$};

  \end{tikzpicture}

  \caption{The  graph $\Gamma_{\pi}$, decorated with the orders of vanishing of  the function $f \circ \pi $  and red arrows corresponding to the components of the polar curve weighted with the intersection numbers.}
  \end{figure}
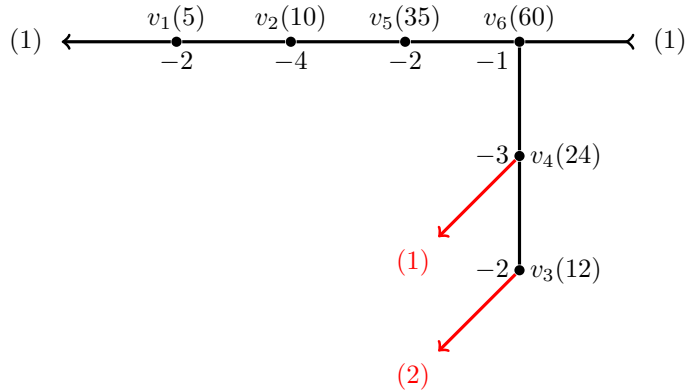

 For $ i=1,\dots 6$, denote $q_i := q_{g,v_i}^f$. By applying Theorem \ref{laplacien} on the dual graph $\Gamma_{\pi}$ we get the following system of equations:

$$ \begin{pmatrix}
-2 & 1 & 0 & 0 & 0 & 0 \\
1 & -4 & 0 & 0 & 1 & 0 \\
0 & 0  & -2 & 1 & 0 & 0\\
0 & 0 & 1 & -3 & 0 & 1 \\
0 & 1 & 0 & 0 & -2 & 1  \\
0 & 0 & 0 & 1 & 1 &-1
\end{pmatrix} 	\cdot \begin{pmatrix}
5 q_1 \\
10 q_2 \\
12 q_3 \\
24 q_4 \\
35 q_5 \\
60 q_6

\end{pmatrix} 	=   \begin{pmatrix}
1-2 \\
2-2 \\
1-2 \\
2-2 \\
2-2 \\
2-2

\end{pmatrix} +\begin{pmatrix}
0 \\
0 \\
0 \\
0 \\
0\\
1

\end{pmatrix} - \begin{pmatrix}
0 \\
0 \\
2 \\
1\\
0 \\
0

\end{pmatrix},$$

whose solution is: $$(q_1,q_2,q_3,q_4,q_5,q_6)=(\frac{1}{5},\frac{1}{10} ,\frac{1}{4},\frac{1}{8},\frac{3}{35},\frac{1}{12})$$

\begin{figure}[H]
\begin{tikzpicture}[node distance=5cm, very thick]

 \tikzstyle{titleVertex}      = [ shape=circle,node distance=4cm]

  \tikzstyle{inVertex}      = [ shape=circle,node distance=1.5cm]
\tikzstyle{Vertex}      = [fill, shape=circle, line cap=round,line join=round,>=triangle 45,scale=.4,font=\scriptsize]
  \tikzstyle{Edge}        = [black]
    \tikzstyle{arrowEdge}        = [black, ->]

 \node[Vertex]      (1)               {};

     \node[Vertex]      (2)  [right       of=1] {};
  \node[Vertex]      (3)  [ right of=2]  {};
  \node[Vertex]      (4)  [right  of=3]  {};
  \node[Vertex]      (5)  [below        of=4] {};
  \node[Vertex]      (6)  [below       of=5] {};

                                    \path 
            (1) edge[Edge] (2)
      (2) edge[Edge] (3)
             (3) edge[Edge] (4)
               (4) edge[Edge] (5)
              
                   (5) edge[Edge] (6);
\draw (1) node[above] {$(\frac{1}{5})$}
          (2) node[above] {$(\frac{1}{10})$}
          (3) node [above]{$(\frac{3}{35})$}
           (4) node [above]{$(\frac{1}{12})$}
           (5) node [left]{$$}
                       (5) node [right]{$(\frac{1}{8})$}
             (6) node [left]{$$}
  (1) node[below] {$$}
          (2) node[below] {$$}
          (3) node [below]{$$}
           (4) node [below left ]{$$}
                (5) node [ left]{$$}
           (6) node [ left ]{$$}
             (6) node [right]{$(\frac{1}{4})$};                                  
  \end{tikzpicture} \caption{The  dual graph weighted with the inner rates $q_i$.} 
  \end{figure}
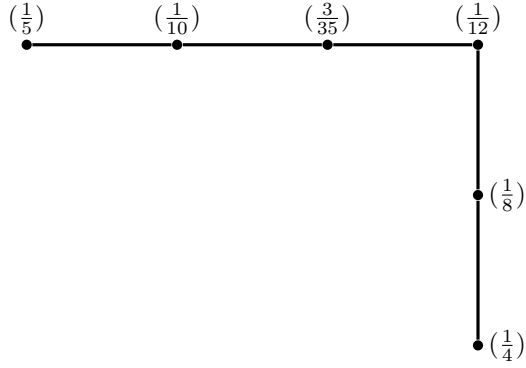
  
  \end{exam}
  \section{Non-archimedean link and metric dual graph}\label{section4}
 In this section we will  follow closely   \cite[Preliminaries]{BFP}. Let $(X,0)$ be a complex surface germ with an isolated singularity. Denote by $\mathcal{O}=\widehat{\mathcal{O}_{X,0}}$ the completion of the local ring of $X$ at $0$ with respect to its maximal ideal.
\begin{defi}
A (rank $1$) \textbf{semivaluation} on $\mathcal{O}$ is a map $v:\mathcal{O} \longrightarrow \mathbb{R} \cup \{ +\infty \} $ such that, for every $f,g \in \mathcal{O}$ and every $\lambda \in \mathbb{C}^{\times}$

\begin{itemize}

\item $v(fg)=v(f)+v(g)$
\item$ v(f+g) \geq \mathrm{min}\{v(f),v(g)\}$
\item $v(\lambda)=\left\{\begin{array}{@{}l@{}}
    + \infty \ \text{if} \ \lambda=0\\
    
     0 \ \ \ \ \text{if} \ \lambda \neq 0
  \end{array}\right.\,$.
\end{itemize}  A \textbf{valuation} is  semivalutation such that $0$ is the only element sent to $+\infty$. \end{defi} 

\begin{exam}  Let $\pi: (X_\pi,E) \longrightarrow (X,0)$ be a good resolution of $(X,0)$. Let $E_v$ be an irreducible component of the exceptional divisor $E$. Let $\mathfrak{M}$ be the maximal ideal of $\mathcal{O}$ and set $m_v(\mathfrak{M})=\mathrm{inf}\{ m_v(f) \ | \ f \in \mathfrak{M}\}$. The map

 $$\begin{array}{rcl}
\mathrm{val}_{E_v}:\mathcal{O}&\to& \mathbb{R}_{+} \cup \{+ \infty \},\\
f &\mapsto &\frac{m_v(f)}{m_v(\mathfrak{M})}
\end{array} $$   is a valuation on $\mathcal{O}$. We call it the \textbf{divisorial valuation} associated with $E_v$.
\end{exam}

\begin{defi}\label{linknonarchimedien}
The \textbf{non-archimedean link} $\mathrm{NL}(X,0)$ of $(X,0)$ is the set $$ \mathrm{NL}(X,0) = \{ v: \mathcal{O} \longrightarrow \mathbb{R}_{+} \cup \{+\infty\}  \ \text{semi-valuation} \ | \ v(\mathfrak{M}) =1 \ \text{and} \ v_{|\mathbb{C}^* } =0  \} $$
whose topology is induced from the product topology $(\mathbb{R}_{+} \cup \{+ \infty \})^{\mathfrak{M}}$.
\end{defi}
Let $\pi: (X_\pi,E) \longrightarrow (X,0)$ be a good resolution of $(X,0)$. There exists an embedding 
$$i_{\pi} :\Gamma_{\pi} \longrightarrow \mathrm{NL}(X,0)$$ and a continuous retraction $$r_{\pi} :\mathrm{NL}(X,0) \longrightarrow \Gamma_{\pi}$$ such that $r_{\pi} \circ i_{\pi} = \mathrm{Id}_{\Gamma_{\pi}}$. The embedding $i_{\pi}$ maps each vertex $v$ of $\Gamma_{\pi}$ to the divisorial valuation associated with the component $E_v$, and each edge $e_{v,v'}$ that corresponds to a point $p$ of the intersection $E_v \cap E_{v'}$ to the set of monomial valuations on $X_{\pi}$ at $p$. We refer to \cite[Preliminaries]{BFP}. 


 \begin{thm}[{See e.g., \cite[Theorem 2.27]{GignacRuggiero2017} and \cite[ Theorem 7.9]{Jonsson2015}}] \label{universaldualgraph}
 The family of the continuous retractions $\{ r_{\pi} \ | \ \pi \ \text{is a good resolution of} \ (X,0) \}$ induces a natural homeomorphism from $\mathrm{NL}(X,0)$ to the inverse limit of the dual graphs $\Gamma_{\pi}$
$$\mathrm{NL}(X,0) \cong \varprojlim_{ \pi } \Gamma_{\pi}. $$
 \end{thm} 

\noindent Let $\Phi=(g,f): (X,0) \longrightarrow (\mathbb{C}^2,0)$ be a finite morphism. It induces a natural morphism $\widetilde{\Phi} : \mathrm{NL}(X,0) \longrightarrow (\mathbb{C}^2,0)$. Indeed, we set
$\Phi^{\#}:\widehat{ \mathcal{O}_{(\mathbb{C}^2,0)}} \longrightarrow \widehat{ \mathcal{O}_{(X,0)}}$ defined by  $\Phi^{\#} (h)=h \circ \Phi$. Hence

$$  \begin{array}{rcl}
\widetilde{\Phi}:\mathrm{NL}(X,0)&\to&  \mathrm{NL}(\mathbb{C}^2,0).\\
v &\mapsto &v \circ \Phi^{\#}
\end{array} $$ By using the description of $\mathrm{NL}(X,0)$ as inverse limit of dual graphs (Theorem \ref{universaldualgraph}), we have a way to compute the image by $\widetilde{\Phi}$ of a divisorial valuation $v$: 
\begin{enumerate}
\item Take a good resolution $\pi: (X_{\pi},E) \longrightarrow (X,0)$ such that $v$ is associated to some exceptional component $E_v$ of $E$.
\item Let $p$ be a smooth point of $E$ such that $p$ does not belong to $f^* \cup g^* \cup \Pi_{\Phi}^*$.
\item Take two curvettes $\gamma_1^*$ and $\gamma_2^*$ of $E_v$ which passes through two distinct points of $E_v$ which are in a neighborhood $O_p$ of $p$.
\item Consider the curves $\gamma_1=\Phi \circ \pi (\gamma_1^*)$ and $\gamma_2=\Phi \circ \pi (\gamma_2^*)$.
\item Let  $\sigma:(Y,F) \longrightarrow (\mathbb{C}^2,0)$ be the minimal sequence of blowups such that the strict transforms of $\gamma_1$ and $\gamma_2$ are smooth and meet transversely an irreducible component $F_{w} \subset F$ at two different points. The morphism $\sigma$ and the exceptional component $F_w$ do not depend of the choice of the curvettes $\gamma_1^*$ and $\gamma_2^*$ as a consequence of Proposition \ref{inner-rate}.
\item Finally, $\widetilde{\Phi}(v)=w$ where $w$ is the divisorial valuation  of $\mathrm{NL}(\mathbb{C}^2,0)$ associated to $F_{w}$.
\end{enumerate}
 Since $\Phi$ is finite, the map $\widetilde{\Phi}$ is a ramified covering.  It is reflected at the level of the dual graphs as follows. Let $\pi:(X_{\pi},E) \longrightarrow (X,0)$ be a good resolution of $(X,0)$ and $\sigma:(Y,F) \longrightarrow (\mathbb{C}^2,0)$ be a sequence of blowups of  the origin of $\mathbb{C}^2$ such that for any divisorial valuation $v$ associated to an irreducible component $E_v$ of $E$,  the induced divisorial valuation $\widetilde{\Phi}(v)=w$ is associated to an irreducible component $F_{w}$ of $F$. Then the set $\{\widetilde{\Phi}(\mathrm{val}_{E_v}) \ | \ E_v \subset E   \}$ defines a subgraph of $\Gamma_{\sigma}$ and the restriction of $\widetilde{\Phi}$ to $\Gamma_{\pi}$ is a ramified covering of graphs onto its image.

Let  $f:(X,0) \longrightarrow (\mathbb{C},0)$ be a holomorphic function germ. We now define a metric on $\mathrm{NL}(X,0)$ compatible with $f$ following \cite[Subsection 2.3]{BFP}. We start by doing it on any graph $\Gamma_{\pi}$ where $\pi:(X_{\pi},E) \longrightarrow (X,0)$ is a good resolution.  Let us endow the dual graph  $\Gamma_{\pi}$ with a metric by declaring the length of an edge $e_{v,v'}$ to be\begin{equation*}\mathrm{length}_f(e_{v,v'})=\frac{1}{m_v(f)m_{v'}(f)}.\end{equation*}(See \cite[Remark 2.5]{BFP} for a geometric interpretation of this metric in terms of the Milnor fibers of $f$).

Now, let $p$ be an intersection  point between two irreducible components  $E_v$ and $E_{v'}$ of $E$ and let $\tilde{\pi}:(X_{\tilde{\pi}},\tilde{E}) \longrightarrow (X,0)$ be the good resolution obtained by blowing up the point $p$. Then the exceptional component $E_w$ that arises has multiplicity $m_w(f)=m_v(f)+m_{v'}(f)$ . Since $\frac{1}{m_w(f)m_v(f)}+\frac{1}{m_w(f)m_{v'}(f)}=\frac{1}{m_v(f)m_{v'}(f)}$,  the inclusion of $\Gamma_{\pi}$ in $\Gamma_{\tilde{\pi}}$ is an isometry. Therefore, by passing to the limit, the metric on the dual graphs $\Gamma_{\pi}$  defines a metric on the non-archimedean link $\mathrm{NL}(X,0)$ called  the \textbf{skeletal metric on $\mathrm{NL}(X,0)$} with respect to $f$.
\begin{exam}
Consider again the dual graph $\Gamma_{\pi}$ of the minimal good resolution $\pi$ of the curve of equation $f(x,y)=y^5-x^{12}=0$ in $\mathbb{C}^2$ introduced in Example \ref{exemple1}. The following  figure shows  $\Gamma_{\pi}$  decorated with the order of vanishing of the function $f \circ \pi$ and the length of each of its edges.

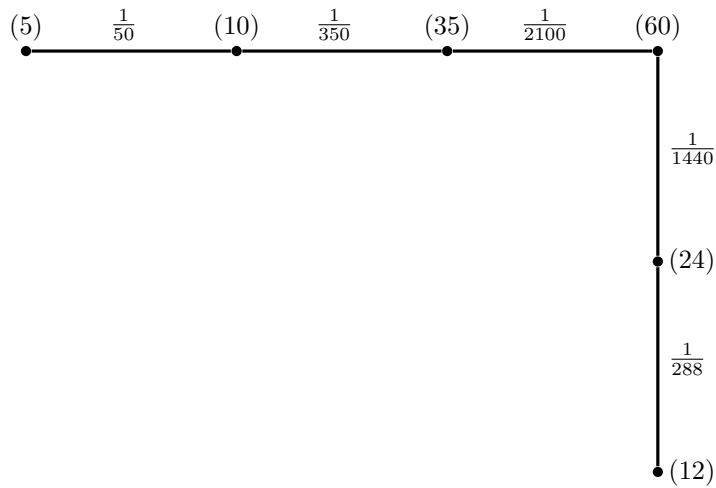
\begin{figure}[H]
\begin{tikzpicture}[node distance=7cm, very thick]
 \tikzstyle{titleVertex}      = [ shape=circle,node distance=4cm] 
 \tikzstyle{inVertex}      = [ shape=circle,node distance=1.5cm]
  \tikzstyle{lengthVertex}      = [ shape=circle,node distance=1.3cm]
  \tikzstyle{Vertex}      = [fill, shape=circle, line cap=round,line join=round,>=triangle 45,scale=.4,font=\scriptsize]
  \tikzstyle{Edge}        = [black]
    \tikzstyle{arrowEdge}        = [black, ->]

 \node[Vertex]      (1)               {};

     \node[Vertex]      (2)  [right       of=1] {};
  \node[Vertex]      (3)  [ right of=2]  {};
  \node[Vertex]      (4)  [right  of=3]  {};
  \node[Vertex]      (5)  [below        of=4] {};
  \node[Vertex]      (6)  [below       of=5] {};
  \node[lengthVertex] (8)  [right     of=1]  {}; 
  \node[lengthVertex] (9)  [right     of=2]  {}; 
  \node[lengthVertex] (10)  [right     of=3]  {}; 
    \node[lengthVertex] (11)  [below     of=4]  {}; 
\node[lengthVertex] (12)  [below     of=5]  {};

        \path 
            (1) edge[Edge]  (2)
      (2) edge[Edge] (3)
             (3) edge[Edge] (4)
               (4) edge[Edge] (5)
               (5) edge[Edge] (6);
\draw (1) node[above] {$(5) $}
          (2) node[above] {$(10)$}
          (3) node [above]{$(35)$}
           (4) node [above]{$(60)$}
            (5) node [right]{$$}
             (6) node [right]{$$}
                 (5) node [right]{$(24)$}
                        
            (6) node [right]{$(12)$}
              (8) node[above]{$\frac{1}{50}$}
    (9) node[above]{$\frac{1}{350}$}
    (10) node[above]{$\frac{1}{2100}$}
                (11) node[right]{$\frac{1}{1440}$}
                (12) node[right]{$\frac{1}{288}$};                 
                  \end{tikzpicture} 
                  \caption{The graph $\Gamma_{\pi}$  decorated with the orders of vanishing of the function $f \circ \pi$ and the length of each of its edges.}   \end{figure}
  
  \end{exam}

  With the metric defined before on $\Gamma_{\pi}$ we then have a metric graph
  $$ \Gamma_{\pi}=\left(V(\Gamma_{\pi}),E(\Gamma_{\pi}), \begin{array}{rcl}
\mathrm{length}_{f}:E(\Gamma_{\pi}) &\to& \mathbb{Q}_{+} \\
e_{v v'} &\mapsto &\mathrm{length}(e_{v v'})=\frac{1}{m_v(f)m_{v'}(f)}
\end{array} \right)$$

\begin{defi}\label{piecewise}
A function $F:\Gamma_{\pi} \longrightarrow \mathbb{R}$ is  \textbf{piecewise linear} if $F$ is a continuous piecewise affine map with integral slopes with respect to the metric induced by $\mathrm{length}_{f}$ and $F$ has only finitely many points of non-linearity on each edge of $\Gamma_{\pi}$.
\end{defi}A divisor $D= \sum_{v \in \Gamma_{\pi}} a_v v$ of $\Gamma_{\pi}$ is a finite sum of points of $\Gamma_{\pi}$ with integral coefficients $a_v \in \mathbb{Z}$.  We denote by $\mathrm{Div}(\Gamma_{\pi})$ the group of divisors of $\Gamma_{\pi}$. 






\begin{defi}\label{metricgraph}
Let $F$ be a piecewise linear map on the dual graph $\Gamma_{\pi}$. The \textbf{Laplacian} of $F$ is the divisor $\Delta_{\Gamma_{\pi}}(F)$ of $\Gamma_{\pi}$ whose coefficient at a point $v$ of $\Gamma_{\pi}$ is the sum of \textbf{the outgoing slopes} of $F$ at $v$: $$ \Delta_{\Gamma_{\pi}}(F)= \sum_{v \in V(\Gamma_{\pi}) }\left( \sum_{v' \neq v} \frac{F(v')-F(v)}{\mathrm{length}_{f}(v',v)} (E_v \cdot E_{v'})\right)  v. $$
\end{defi}

\section{The inner rates function}\label{section5}
Let $(X,0)$ be a complex surface germ with an isolated singularity and let $g,f :(X,0) \longrightarrow (\mathbb{C},0)$ be two holomorphic functions such that the morphism $\Phi=(g,f) :(X,0) \longrightarrow (\mathbb{C}^2,0)$ is finite. The aim of this section is to define a continuous function on the non archimedean link  whose values are the inner rates on the divisorial elements of $\mathrm{NL}(X,0)$.

\noindent Now, let us state the main proposition of this section which is a direct adaptation of \cite[Lemma 3.8]{BFP}.
\begin{prop}\label{inner-rate function} Denote by $(u_1,u_2)=(g,f)$  the coordinates of $\mathbb{C}^2$. There exists a unique continuous function  $$ \mathcal{I}_{g}^{f} : \mathrm{NL}(X,0) \longrightarrow \mathbb{R}_{>0} \cup \{ +\infty  \} $$ such that $\mathcal{I}_g^f(v)=q_{g, v}^{f}$  for every divisorial point $v$ of $\mathrm{NL}(X,0)$  and such that we have the following commutative diagram
 
 $$
\xymatrix{
  \ar[rd]^{\mathcal{I}_g^f }  \mathrm{NL}(X,0)\ar[r]^{\widetilde{\Phi}} &\mathrm{NL}(\mathbb{C}^2,0) \ar[d]^{\mathcal{I}_{u_1}^{u_2}} \\
  & \mathbb{R}_{>0} \cup \{ +\infty  \}  
}$$
 Furthermore, if $\pi$ is a good resolution of $(X,0)$ compatible with the finite morphism $\Phi=(g,f)$ then $\mathcal{I}_g^f$  is piecewise linear on $\Gamma_{\pi}$ with respect to the metric $\mathrm{length}_f$.
\end{prop}

\begin{defi}
The function $\mathcal{I}_g^f$ is called \textbf{the inner rates function with respect to the morphism} $\Phi=(g,f)$.
\end{defi}

\noindent In order to prove Proposition  \ref{inner-rate function} we will need the following lemmas.

\begin{lemm}\label{recurence'}
Let $\Phi=(g,f): (X,0) \longrightarrow (\mathbb{C}^2,0)$ be a finite morphism and let $\pi :(X_{\pi},E) \longrightarrow (X,0)$ be a good resolution of $(X,0)$  and $(gf)^{-1}(0)$. Let  $E_v$  be an irreducible component  of $E$  such that it intersects the curve $f^* \cup g^*$. We have  
$$ q_{g,v}^f=\frac{m_v(g)}{m_v(f)}$$
\end{lemm}
\begin{proof} Let $p$ be a point of $g^*$. Let $(x,y)$ be a local system of coordinates centered at $p$ such that $E_v$ has  local  equation $x=0$ and such that $(f \circ \pi)(x,y)=x^{m_v(f)}$. Let  $S_g$  be the element of  $\mathbb{C}\{ x,y\}$ such that $ (g \circ \pi)(x,y) = x^{m_v(g)}S_g(x,y)$ . Then, the equation of the strict transform $g^*$ in the coordinates $(x,y)$ is  $S_g(x,y)=0$. Since the curve  $g^*$  is transverse to  the component $E_v$  we can write $ S_g(x,y)=y^{k_v}U_g(x,y)$  where $U_g$  is a unit of $\mathbb{C}\{x,y\}$ and $k_v$ is an element of   $\mathbb{N}^*$.

 Let $\gamma_{1}^*$ and $\gamma_{2}^*$ be two curvettes of $E_v$ parametrized respectively by
$$ t \mapsto (t,\alpha), \ t \mapsto (t,\beta ), \ \alpha,\beta \in \mathbb{C}. $$
 The curves $\gamma_1= (\Phi \circ \pi)(\gamma_1^*) $ and $\gamma_2=(\Phi \circ \pi)(\gamma_2^*)$  are then parametrized respectively by
$$  t \mapsto (\alpha^{k_v}  t^{m_v(g)} U_g(t,\alpha)  ,t^{m_v(f)})$$
$$  t \mapsto (\beta^{k_v} t^{m_v(g)} U_g(t,\beta)  ,t^{m_v(f)}).$$
Therefore, for $\epsilon > 0$, we have  $$ \mathrm{d}(\gamma_1 \cap \{ u_2 = \epsilon \}, \gamma_2 \cap \{ u_2 =\epsilon\})= \displaystyle\left\lvert \epsilon^{\frac{m_v(g)}{m_v(f)}}  \displaystyle\right\rvert H(\epsilon),  $$
where
$$H(\epsilon)= \displaystyle\left\lvert  \alpha^{k_v} U_g(\epsilon^{\frac{1}{m_v(f)}} ,\alpha) - \beta^{k_v} U_g(\epsilon^{\frac{1}{m_v(f)}},\beta  )     \displaystyle\right\rvert.$$ There exists two numbers $\alpha$ and $\beta$ such that $H(0) \neq 0$, we can then conclude that $q_{g,v}^f=\frac{m_v(g)}{m_v(f)}$. The case $p \in f^*$ is treated with similar arguments.

\end{proof}
\begin{defi}
A good resolution $\pi:(X_{\pi},E) \longrightarrow (X,0)$ of $(X,0)$  is said to be \textbf{compatible with the finite morphism $\Phi=(g,f)$} if it is a good resolution of the curve $\Pi_{\Phi} \cup (gf)^{-1}(0)$. \end{defi}

\begin{lemm}\label{recurence}

Let $\Phi=(g,f): (X,0) \longrightarrow (\mathbb{C}^2,0)$ be a finite morphism and let $\pi :(X_{\pi},E) \longrightarrow (X,0)$ be a good resolution of $(X,0)$ compatible with the finite morphism $\Phi$. Let $E_v$ be an irreducible component of $E$. Let $p$ be a point of  $E_v$ and let $E_{w}$ be the exceptional component created by the blowup of $X_{\pi}$ at $p$. Then:
\begin{enumerate}

\item if $p$ is a smooth point of $E$, then

\begin{itemize}
\item if $p$ does not lies in the strict transforms $g^*,f^*$ or $\Pi_{\Phi}^*$, then

                 \begin{align*}
		m_w(g)=m_v(g), && m_w(f)=m_v(f),  &&  \text{and} && q_{g,w}^f=q_{g,v}^f+\frac{1}{m_v(f)};
		\end{align*}
		
		\item if $p$ lies in the strict transform  $g^*$ then:
		
		\begin{align*}
		m_w(g)=m_v(g)+g^*\cdot E_v, && m_w(f)=m_v(f) ,&& \text{and} && q_{g,w}^f=\frac{m_w(g)}{m_w(f)}=q_{g,v}^{f} + \frac{g^* \cdot E_v}{m_v(f)};		\end{align*}	

                 \item if $p$ lies in the strict transform  $f^*$ then:

                	\begin{align*}
		m_w(g)=m_v(g), &&m_w(f)=m_v(f)+f^*\cdot E_v, && \text{and}&& q_{g,w}^f=\frac{m_w(g)}{m_w(f)}=\frac{m_v(f)q_{g,v}^f}{m_v(f)+ f^*  \cdot E_v};
		\end{align*}

		\item if $p$ lies in the strict transform  $\Pi_{\Phi}^*$
		\begin{align*}
		m_w(g)=m_v(g), &&  m_w(f)=m_v(f),&& \text{and}&& q_{g,w}^f=q_{g,v}^f+\frac{1+ \Pi_{\Phi}^*\cdot E_v}{m_v(f)};
		\end{align*}

\end{itemize}
				
\item if $p$ is a double point of $E_v$ and $E_{v'}$, then
		\begin{align*}
		m_w(g)=m_v(g)+ m_{v'}(g) &&\text{and}&& m_w(f)=m_v(f) + m_{v'}(f);
		\end{align*}
		$$q_{g,w}^f=\frac{q_{g,v}^f m_v(f) + q_{g,v'}^f m_{v'}(f)}{m_v(f)+m_{v'}(f)}$$
		
\end{enumerate}
\end{lemm}

\begin{rema}
The point $(i)$ of this lemma is more general than needed for this paper. Indeed, we will only use it for the morphism identity of $\mathbb{C}^2$, whose polar curve is empty.
\end{rema}
\begin{proof}Denote by $e_p$ the blowup of $X_{\pi}$ centered at $p$. Assume first that $p$ is a smooth point of $E_v$ which does not lie in the strict transform of $f,$ $g$ or the polar curve. By remark \ref{remarque} there exists a local system of coordinates $(x,y)$ centered at $p$ such that:

 $$\left \{
\begin{array}{rcl}
(g \circ \pi)(x,y)&=&x^{m_v(g)} g_{0}(x) +x^{ q_{g,v}^{f} m_v(f) }\sum_{j \geq 1 }  y^j g_{j}(x) \  \\
(f \circ \pi )(x,y)&=&x^{m_v(f)}
\end{array}
\right.
$$ with   $g_{0}(0) \neq 0$,  $g_{1}(0) \neq 0$.\\
Let $e_p$ be the blowup of $X_{\pi}$ at $p$. In the coordinates chart $ (x,y) \mapsto (x',x'y')$ we have
$$\left \{
\begin{array}{rcl}
(g \circ \pi \circ e_p) (x',y')&=&{x'}^{m_v(g)} g_{0}({x'}) +{x'}^{q_{g,v}^f m_v(f)+1 }\sum_{j \geq 1 }  {y'}^j {x'}^{j-1} g_{j}({x'}) \\
(f \circ \pi \circ e_p) (x',y')&=&{x'}^{m_v(f)}.
\end{array}
\right.
$$Therefore $m_w(g)=m_v(g)$ and $m_w(f)=m_v(f)$. By comparing the equality of $g \circ \pi\circ e_p$ with the first equality given by the remark \ref{remarque} for $E_w$, we get  $q_{g,w}^f m_w(f)=q_{g,v}^f m_v(f)+1$, which proves 
$$q_{g,w}^f=q_{g,v}^f +\frac{1}{m_v(f)}.$$

Assume that $p$ is in $g^*$. Using again the notations of the proof of Lemma \ref{recurence'}.  In the coordinates chart $ (x,y) \mapsto (x',x'y'),$ we have:
\begin{eqnarray*} \label{1}
\left \{
\begin{array}{rcl}
(g \circ \pi  \circ e_p)(x',y'))&=&{x'}^{m_v(g)+k_v}y'^{k_v} U_g(x',x'y')\\
(f \circ \pi  \circ e_p) (x',y'))&=&{x'}^{m_v(f)}
\end{array}
\right. 
\end{eqnarray*}
 where $U_g$ is a unit of $\mathbb{C}\{x,y\}$ and $k_v$ is a non zero integer. Thus, by noticing that $k_v=g^*\cdot E_v$, we have  \begin{align*}
		m_w(g)=m_v(g)+g^* \cdot E_v &&\text{and}&& m_w(f)=m_v(f). 
		\end{align*} Since the component $E_w$  intersects $g^*$, then, by the lemma \ref{recurence'}, we get that $$ q_{g,w}^f=\frac{m_w(g)}{m_w(f)}=q_{g,v}^{f} + \frac{g^* \cdot E_v}{m_v(f)}$$

		 If $p$ is in $f^*$, with similar arguments we get:   	
		 \begin{align*}
		m_w(g)=m_v(g) &&\text{and}&& m_w(f)=m_v(f)+f^*\cdot E_v.  	
			\end{align*}	
			$$q_{g,w}^f=\frac{m_w(g)}{m_w(f)}=\frac{m_v(f)q_{g,v}^f}{m_v(f)+ f^*  \cdot E_v}$$

		We will now suppose that $p$ lies in the strict transform of the polar curve $\Pi_{\Phi}^{*}$.  Let $(x,y)$ be a local system of coordinates centered at $p$ such that $E_v$ has  local  equation $x=0$ and such that $(f \circ \pi)(x,y)=x^{m_v(f)}$. Let  $U$ be the  unit of $\mathbb{C}\{ x,y\}$ such that $ (g \circ \pi)(x,y) = x^{m_v(g)}U(x,y)$.
Let $e_p$ be the blowup of $X_{\pi}$ at $p$. in the coordinates chart $ (x,y) \mapsto (x',x'y')$ we have
$$\left \{
\begin{array}{rcl}
(g  \circ \pi \circ e_p )(x',y')&=&{x'}^{m_v(g)} U(x',x'y')\\
(f  \circ \pi \circ e_p) (x',y')&=&{x'}^{m_v(f)}.
\end{array}
\right.
$$
We then deduce that $m_w(f)=m_v(f)$, $m_w(g)=m_v(g)$ and it follows using the inner rates formula \ref{laplacien} on the vertex $w$ that:

$$q_{g,w}^f=q_{g,v}^f+\frac{1+\Pi_{\Phi}^* \cdot E_w}{m_v(f)}.$$ We now need to prove that $\Pi_{\Phi}^* \cdot E_w = \Pi_{\Phi}^* \cdot E_v$. The equation of the total transform of the polar curve in the coordinates $(x,y)$ is $x^{m_v(g)+m_v(f)-1}\frac{\partial U}{\partial y}(x,y)=0$. Since the polar curve is transverse to the exceptional divisor $E_v$ we can write $\frac{\partial U}{\partial y}(x,y)=y^{k_v}V(x,y)$ where $V$ is an element of $\mathbb{C}\{x,y\}$ such that $y$ is not a divisor of $V$  and $k_v$ is a non zero positive integer. The total transform of the polar curve by $e_p \circ \pi$, in the coordinates  chart $ (x,y) \mapsto (x',x'y')$, is:
$$m_v(f)x'^{m_v(g)+m_v(f)+k_{v}-1}y'^{k_v} V(x',x'y')=0.$$
Using this last equation, a direct computation shows that the intersection number between the curve  $E_w$ and the strict transform of the polar curve equals $k_v=\Pi_{\Phi}^* \cdot E_v$.\\

Assume now that $p$ is an intersecting point between two irreducible components $E_v$ and $E_{v'}$. In local coordinates $(x,y)$ centered at $p$, we can assume without loss of generality that
In the coordinates chart $ (x,y) \mapsto (x',x'y')$ we have

$$(f \circ \pi \circ e_p ) (x',y') = x'^{m_v(f)+m_v(g)} y'^{m_{v'}(f)},$$

\noindent we then deduce that $m_w(f)=m_v(f)+m_{v'}(f)$. By symmetry of the roles of $f$ and $g$  we also have  $m_w(g)=m_v(g)+m_{v'}(g)$. By applying the inner rates formula  (Theorem \ref{laplacien}) on the vertex $w$ we get:
$$q_{g,w}^f=\frac{q_{g,v}^f m_v(f)+q_{g,v'}^f m_{v'}(f)}{m_v(f)+m_{v'}(f)}$$\end{proof}


Now we are ready to prove the proposition that shows the existence of the inner rates function.
\begin{proof}[Proof of Proposition \ref{inner-rate function}]

Let $\pi:(X_{\pi},E) \longrightarrow (X,0) $ be a good resolution of $(X,0)$ compatible with the morphism $\Phi=(g,f)$.
We only need to show that the inner rates extend uniquely to a continuous map on $\Gamma_{\pi}$ which is linear on its edges with integral slopes, because $\mathrm{NL}(X,0)$ is homeomorphic to the inverse limit of dual graphs.\\
 Let $e_{v v'}$ be an edge of $\Gamma_{\pi} $. By definition, the  slope on $e_{v,v'}$ is $ \frac{q_{g,v' }^f - q_{g,v }^f  }{\mathrm{length}_f(e_{v v' } )}= m_v(f)m_{v'}(f)(q_{g ,v' }^f - q_{g,v }^f )$ and it  is an integer by Proposition  \ref{inner-rate}.
 
  By density, it is sufficient to prove the linearity of $\mathcal{I}_g^f$ along  $e_{v v'}$ on the set of the divisorial valuations. On the other hand, since every divisorial point of that edge are obtained by successively  blowing up  its double points, it is sufficient to prove that $\mathcal{I}_g^f$  is linear on the set $\{ v, v',v''\}$, where $v''$ is obtained by blowing up the double point $E_{v} \cap E_{v'}$. Therefore, all we have to show is that 
\begin{equation}\label{funk} \frac{\mathcal{I}_g^f(v') - \mathcal{I}_g^f(v)}{ \mathrm{length}_f(e_{v v'} ) } =\frac{\mathcal{I}_g^f(v'') - \mathcal{I}_g^f(v)}{ \mathrm{length}_f(e_{v v''} ) }.\end{equation} Since the good resolution $\pi$ is compatible with the morphism $\Phi$, the strict transforms  $f^*,g^*$ and $\Pi_{\Phi}^*$ do not pass through the double point $E_v \cap E_{v'}$. The equality \eqref{funk}  is then a direct consequence of the last part of Lemma \ref{recurence}. 

\end{proof}

\subsection{Laplacian formula on the metric dual graph of a good resolution}

Let  $(X,0)$ be a complex surface germ with an isolated singularity and $\Phi=(g,f) : (X,0) \longrightarrow (\mathbb{C}^2,0)$ be a finite morphism. Let $\pi: (X_{\pi},E) \longrightarrow (X,0)$ be a good resolution of $(X,0)$ compatible with the morphism $\Phi$. We proved in Proposition \ref{inner-rate function} that the inner rate function $\mathcal{I}_g^f$ is a piecewise linear map on the dual graph $\Gamma_{\pi}$ with respect to the metric $\mathrm{Length}_f$. Then, as an application of the inner rates formula \ref{laplacien} we will compute the laplacian of $\mathcal{I}_g^f$ on the dual graph. We will now give three  example of  divisor on the dual graph $\Gamma_{\pi}$ which are important for the statement of our result :
\begin{enumerate}

\item $K_{\Gamma_{\pi}} = \sum_{v \in \Gamma_{\pi} }  m_v(f)(\mathrm{val}_{\Gamma_{\pi}}(v)+2g_v-2).v$ it is a divisor because every point $v$ which is not a vertex of $\Gamma_{\pi}$ has valency $2$ and genus equal $0$.
\item $F_{\Gamma_{\pi}} = \sum_{v \in \Gamma_{\pi}} (m_v(f)+ m_v(g) )(f^* \cdot E_v) .v$ it is a divisor because there are finitely many branches of $f^*$.
\item $P_{\Gamma_{\pi}} = \sum_{v \in \Gamma_{\pi} }m_v(f) (\Pi_{\Phi}^* \cdot E_v).v $ it is a divisor because there are  finitely many branches of $\Pi_{\Phi}^*$. 
\end{enumerate}

\noindent Now, we are ready to state the inner rates formula (Theorem \ref{laplacien}) in terms of the laplacian of the function $\mathcal{I}_g^f$.
\begin{coro}\label{laplace}
The laplacian of the  inner rates function $\mathcal{I}_g^f$ is given by the formula $$\Delta_{\Gamma_{\pi}}(\mathcal{I}_g^f)= K_{\Gamma_{\pi}} + F_{\Gamma_{\pi}} - P_{\Gamma_{\pi} }$$
\end{coro}

\begin{proof}
Let $v$ be a vertex of $\Gamma_{\pi}$ and let us write the inner rates formula  \ref{laplacien} on $v$: 

$$
 \left( \sum_{i \in V(\Gamma_{\pi})} m_{i}(f)q_{g,i}^f E_i  \right) \cdot E_{v}= \mathrm{val}_{\Gamma_{\pi}}(v) +f^* \cdot E_v-\Pi_{\Phi}^* \cdot E_v+2g_v-2.
$$

 \noindent Let $v_1,v_2,\ldots,v_k$ be the  vertices of $\Gamma_{\pi}$ adjacent to $v$. Then, the last equality is equivalent to $$
  \sum_{i = 1}^{k} m_{v_i}(f)q_{g,v_i}^f +m_{v}(f)q_{g,v}^f E_v^2 = \mathrm{val}_{\Gamma_{\pi}}(v)  +f^* \cdot E_v-\Pi_{\Phi}^* \cdot E_v+2g_v-2.
$$

\noindent By Proposition \ref{laufer} we have $m_v(f)E_v^2=-\sum_{i=1}^{k}m_{v_i}(f)E_{v_i}-f^*\cdot E_v$ and we inject it on the last equation:
$$ \sum_{i=1}^{k} m_{v_i}(f) (q_{g, v_i}^f -q_{g, v }^f ) = (q_{g, v}^f+1) f^* \cdot E_v + \mathrm{val}_{\Gamma_{\pi}}(v) +2g_v-2 - \Pi_{\Phi}^* \cdot E_v.$$
If $f^* \cdot E_v \neq 0$,  by Lemma \ref{recurence'}, we have  $q_{g,v}^f =\frac{m_v(g)}{m_v(f)}$. Then, by multiplying both sides of the equation by $m_v(f)$, we get
$$ \sum_{i=1}^{k} m_v(f) m_{v_i}(f)(q_{g,v_i}^f-q_{g ,v }^f ) = (m_v(g) +m_v(f) ) f^* \cdot E_v + m_v(f) (\mathrm{val}_{\Gamma_{\pi}}(v) +2g_v-2 )- m_v(f)  \Pi_{\Phi}^* \cdot E_v.$$
Since $\length_f(v,v_i) = \frac{1}{m_{v}(f)m_{v_i}(f) }$, we can then conclude that
$$\Delta_{\Gamma_{\pi}}(\mathcal{I}_g^f)= K_{\Gamma_{\pi}} + F_{\Gamma_{\pi}} - P_{\Gamma_{\pi} }$$
\end{proof}

\section{Growth behaviour of the inner rates function}\label{section6}
The aim of this section is to prove Point $(i)$ of Theorem \ref{thmD}.  For this aim, we will need the following lemma which is an adaptation of \cite[Subsection 4.5]{BFP}. 
\begin{lemm}\label{adapted}Let $(X,0)$ be a complex surface germ with an isolated singularity and let $\Phi=(g,f):(X,0) \longrightarrow (\mathbb{C}^2,0)$ be a finite morphism.  Let $\pi:(X_{\pi},E) \longrightarrow (X,0)$ be a good resolution of $(X,0)$. There exists a finite sequence of blowups $\sigma : (Y,F)\longrightarrow (\mathbb{C}^2,0) $, a good resolution $\tilde{\pi}:(X_{\tilde{\pi}},\tilde{E})\longrightarrow (X,0) $, a finite sequence of blowups $\alpha_{\pi}: (X_{\tilde{\pi}},\tilde{E})\longrightarrow (X_{\pi},E) $ and a morphism $\Phi_{\tilde{\pi} }: (X_{\tilde{\pi}},\tilde{E})\longrightarrow (Y,F) $ such that the following diagram commutes  
$$
\xymatrix{
  \ar[rd]^{\Phi_{\tilde{\pi}} }  (X_{\tilde{\pi}},\tilde{E})  \ar@/^1.2pc/[rr]^{\tilde{\pi}}     \ar[r]^{\alpha_{\pi} } &(X_{\pi},E) \ar[r]^{\pi} &  (X,0) \ar[d]^{\Phi}\\
  & (Y,F)  \ar[r]^{\sigma}    & (\mathbb{C}^2,0)
}$$
and the following properties are satisfied 
 \begin{enumerate}
\item For every vertex $v$ of $\Gamma_{\pi}$, we have $\widetilde{\Phi}(v) \in V(\Gamma_{\sigma})$
\item For every vertex $w$ of $\Gamma_{\sigma}$, we have $\widetilde{\Phi}^{-1}(w) \subset V(\Gamma_{\tilde{\pi}} )$.\\
\end{enumerate}
 \end{lemm}
\begin{defi}
We call the map $\tilde{\pi}$ defined as above the good resolution of $(X,0)$ \textbf{adapted}  to $\pi$ and $\Phi$.
\end{defi}
\begin{rema}
This definition is a generalization of \cite[Definition 4.13]{BFP}.
\end{rema}
\begin{proof}[Proof of Lemma \ref{adapted}]
   Let  $\sigma: (Y,F) \longrightarrow (\mathbb{C}^2,0)$ be the minimal sequence of blowups of the origin of $\mathbb{C}^2$  such that $\widetilde{\Phi}(V(\Gamma_{\pi}) )\subset  V(\Gamma_{\sigma})$. Consider the  pullback by $\Phi$  of  the resolution $\sigma$  i.e., 
   $$ X'=\overline{ \{(a,b) \in X \backslash 0 \times Y \backslash F \ | \ \Phi(a)=\sigma(b)  \} } \subset X \times Y.$$ Let $p_1: X' \longrightarrow X$ and $p_2: X' \longrightarrow Y$ be the first and second projection respectively.\\
Let $\pi':X_{\tilde{\pi}} \longrightarrow X'$ be the minimal good resolution of $X'$. We consider the morphisms $$\tilde{\pi} := p_1 \circ \pi' : (X_{\tilde{\pi}},\tilde{E}) \longrightarrow (X,0)    \ \ \text{and} \ \        \Phi_{\tilde{\pi}}:=  p_2 \circ \pi': (X_{\tilde{\pi}},\tilde{E}) \longrightarrow (Y,F). $$ The existence of the morphism $\Phi_{\tilde{\pi}}$ implies that for every vertex $w$ of $\Gamma_{\sigma}$ we have $\widetilde{\Phi}^{-1}(w) \subset V(\Gamma_{\tilde{\pi}} )$. Let $v  \in \Gamma_{\pi}$. By definition of $\sigma$, we have  $\widetilde{\Phi}(v) \subset V(\Gamma_{\sigma})$ and then  $v \in V(\Gamma_{\tilde{\pi}})$. Therefore, $V(\Gamma_{\pi}) \subset V(\Gamma_{\tilde{\pi}})$, which means that the resolution $\tilde{\pi}$ factors through $\pi$ . The resulting sequence of blowups is the desired  $\alpha_{\pi}:(X_{\tilde{\pi}},\tilde{E}) \longrightarrow (X_{\pi},E) $
    \end{proof}

We are now ready to prove Point $(i)$ of Theorem \ref{thmD} that we restate in terms of the inner rates function:
 \begin{prop}\label{property}
Let $\pi :(X_{\pi},E) \longrightarrow (X,0)$ be a good resolution of  the complex surface germ $(X,0)$. Let $v$ be  an element of $V(\Gamma_{\pi}) $. There exists a path  from an $f$-node to $v$  in $\Gamma_{\pi}$  along which the function $\mathcal{I}_g^f$ is strictly increasing.

\end{prop}
\begin{rema}This result is a generalization of \cite[Proposition 3.9]{BFP} which treated the case where the morphism $\Phi$ is a linear projection on $\mathbb{C}^2$.
\end{rema}
\begin{proof}The case where $(X,0)=(\mathbb{C}^2,0)$ and $\Phi$ is the identity of $\mathbb{C}^2$ and where $\pi$ is a finite sequence of point blowups is a direct consequence of Lemma \ref{recurence}.

Let us now treat the general case. The statement when $v$ is an $f$-node is trivial by taking the constant path. Assume now that $v$ is not an $f$-node. Let $\tilde{\pi}:(X_{\tilde{\pi}},\tilde{E}) \longrightarrow (X,0)$ be the good resolution of $(X,0)$ adapted to $\pi$, we will use the notations of lemma \ref{adapted}. Let $\tilde{v} \in V(\Gamma_{\tilde{\pi}})$ such that $\alpha_{\pi}(\tilde{v})=v$. Since $v$ is not an $f$-node,   it is also the case for $\tilde{v}$. Let $(u_1,u_2)$ be the coordinates of $\mathbb{C}^2$. Consider $w=\widetilde{\Phi}(\tilde{v})$ and let  $\beta'$ be the simple path in $\Gamma_{\sigma}$ starting from $u_2$-node  and ending at $w$ .   By the particular case $(X,0)=(\mathbb{C}^2,0)$ treated at the beginning of the proof, the function  $\mathcal{I}_{u_1}^{u_2}$ is strictly increasing  along $\beta'$. The path $\beta'$ can be lifted via the ramified covering $\widetilde{\Phi}$ to a union of simple paths each of them  joining an $f$-node to a vertex of $\widetilde{\Phi}^{-1}(w) \subset \Gamma_{\tilde{\pi}}$. Let us choose one of them ending at $\tilde{v}$ and name it $\tilde{\beta}$. Let $\widetilde{ \alpha}_{\pi} : \Gamma_{\tilde{\pi} } \longrightarrow \Gamma_{\pi}$ be the continuous map induced by  $\alpha_{\pi}$. Then, the function $\mathcal{I}_{g}^{f}$ is strictly increasing along the path $\beta =\tilde{\alpha}_{\pi}(\tilde{\beta} )$ which start from an  $f$-node  and ends at $v$.
\end{proof}

\section{Proof of  points $(ii)$ and $(iii)$ of Theorem \ref{thmD}}\label{section7}

Let $(X,0)$ be a complex surface germ with an isolated singularity and $\Phi=(g,f):(X,0) \longrightarrow (\mathbb{C}^2,0)$ be a finite morphism. We start by restating Theorem \ref{thmD} in terms of the inner rates function defined in Section \ref{section5} and the Hironaka  quotients function on $\mathrm{NL}(X,0)$ that we define now. \begin{defi}[{\cite[Definition 4]{MaugendreMichel2017}}]\label{defihironaka}
Let $E_v$ be an irreducible component of $E$. The \textbf{Hironaka quotient} of $E_v$ with respect to $\Phi=(g,f)$  is the number:

$$ h_{g, v}^f :=\frac{ m_{v}(g) }{m_{v}(f)}.$$
\end{defi}
Using Lemma  \ref{recurence} and the same arguments of the proof of Proposition \ref{inner-rate function} we obtain:
\begin{prop}\label{Hironaka function}
There exists a unique continuous function  $$ \mathcal{H}_{g}^{f} : \mathrm{NL}(X,0) \longrightarrow \mathbb{R}_{>0} \cup \infty  $$ such that $\mathcal{H}_g^f(v)=h_{g, v}^{f}$  for every divisorial point $v$ of $\mathrm{NL}(X,0)$  and we have 
  the following commutative diagram
 
 $$
\xymatrix{
  \ar[rd]^{\mathcal{H}_g^f }  \mathrm{NL}(X,0)\ar[r]^{\widetilde{\Phi}} &\mathrm{NL}(\mathbb{C}^2,0) \ar[d]^{\mathcal{H}_{u_1}^{u_2}} \\
  & \mathbb{R}_{>0} \cup \infty 
}$$
 If $\pi$ is a good resolution of $(X,0)$ compatible with the finite morphism $\Phi=(g,f)$ then $\mathcal{H}_g^f$  is linear on the edges of $\Gamma_{\pi}$ with integral slopes.
\end{prop} 
\begin{proof}
The commutative diagram of this proposition is a consequence of \cite[Remark 5]{MaugendreMichel2017}. We get the rest with the arguments of the proof of Proposition \ref{inner-rate function}\end{proof}
\begin{defi}
The function $\mathcal{H}_g^f$ is called the \textbf{Hironaka quotients function} with respect to the morphism $\Phi=(g,f)$.
\end{defi}



\begin{defi}
We denote by $\mathcal{A}_{\pi}$   the union of every simple paths joining an $f$-node to a $g$-node along which the inner rates function $\mathcal{I}_{g}^{f}$ is strictly increasing.
\end{defi}
Now, we can restate Theorem \ref{thmD} in terms of the inner rates functions and the Hironaka quotients function.

\begin{thm}\label{touten1}
Let $(X,0)$ be a complex surface germ with an isolated singularity and let $\Phi=(g,f)=(u_1,u_2):(X,0) \longrightarrow (\mathbb{C}^2,0)$ be a finite morphism. Let $\pi:(X_{\pi},E) \longrightarrow (X,0)$ be a good resolution of $(X,0)$. Then:
\begin{enumerate}

\item Let $v$ be a vertex of $\Gamma_{\pi}$. There exists a path from an f-node to v in $\Gamma_{\pi}$ along which the function $\mathcal{I}_g^f$ is strictly increasing.

\item The functions $\mathcal{I}_g^f$ and $\mathcal{H}_g^f$  coincide on  $\mathcal{A}_{\pi}$.

\item The Hironaka quotients function $\mathcal{H}_g^f$ restricted to $\Gamma_{\pi}$ is constant on every connected component of the topological closure of $\Gamma_{\pi} \backslash \mathcal{A}_{\pi}$.
\end{enumerate}
\end{thm}
\begin{rema}
Point  $(i)$ is exactly Proposition \ref{property}. It implies in particular that the inner rates function is not locally constant on $\Gamma_{\pi}$.
\end{rema}

The rest of this section is devoted to the proof of Points $(ii)$ and $(iii)$.

\begin{proof}[Proof of Point $(ii)$ of Theorem \ref{touten1}]The case where $(X,0) = (\mathbb{C}^2,0)$ and $\Phi$ is  the  identity  map of $(\mathbb{C}^2,0)$ is a direct consequence of  Lemmas \ref{recurence} and \ref{recurence'}. Notice that in this case $\mathcal{A}_{\pi}$ is a simple path joining the $u_2$-node to the $u_1$-node where $(u_1,u_2)$ are the coordinates of $\mathbb{C}^2$.

Let $\tilde{\pi}:(X_{\tilde{\pi}},\tilde{E}) \longrightarrow (X,0)$ be the good resolution of $(X,0)$ adapted to $\pi$, we use again the commutative diagram of Lemma \ref{adapted}. We  need  to prove that  $\widetilde{\Phi}(\mathcal{A}_{\pi} ) \subset \mathcal{A}_{\sigma} $. Let $\alpha: [0,1] \longrightarrow \mathcal{A}_{\tilde{\pi}} $ be a path from an $f$-node to a $g$-node along which the inner rates function $\mathcal{I}_g^f$ is strictly increasing. The existence of the morphism $\Phi_{\tilde{\pi}} $ implies that every $f$-node (resp. $g$-node) is sent by $\widetilde{\Phi}$ to the $u_2$-node (resp. $u_1$-node) of $\Gamma_{\sigma}$. Thus, $\beta(t)=\widetilde{\Phi}(\alpha(t))$ is a path starting from the $u_2$-node and ending in the $u_1$-node. By Proposition \ref{inner-rate function}  we have the following diagram $$\xymatrix{
  \ar[rd]^{\mathcal{I}_g^f }  \mathrm{NL}(X,0)\ar[r]^{\widetilde{\Phi}} &\mathrm{NL}(\mathbb{C}^2,0) \ar[d]^{\mathcal{I}_{u_1}^{u_2}} \\
  & \mathbb{R}_{>0} \cup \{+\infty \}
}$$ Then $\mathcal{I}_{u_1}^{u_2} ( \beta(t)) =( \mathcal{I}_{u_1}^{u_2} \circ \widetilde{\Phi})(\alpha(t)) =\mathcal{I}_{g}^{f}(\alpha(t))$ which means that the function $\mathcal{I}_{u_1}^{u_2}(\beta(t))$ is strictly increasing. By definition, the curve $\widetilde{\Phi}(\alpha)=\beta$ is included in $\mathcal{A}_{\sigma}$  and it follows that $\widetilde{\Phi}(\mathcal{A}_{\tilde{\pi}} ) \subset \mathcal{A}_{\sigma}$. Since  $\alpha_{\pi}$  is a sequence of blowups of $X_\pi$ then $ \mathcal{A}_{\pi} \subset \mathcal{A}_{\tilde{\pi}} $ which implies that $\widetilde{\Phi}(\mathcal{A}_{\pi}) \subset \widetilde{\Phi}(\mathcal{A}_{\tilde{\pi}} ) \subset \mathcal{A}_{\sigma}$. Now, let $v \in \mathcal{A}_{\pi}$, we know that the functions $\mathcal{I}_{u_1}^{u_2}$ and $ \mathcal{H}_{u_1}^{u_2}$ coincide on $\mathcal{A}_{\sigma} $ which implies that \begin{equation}\label{hironakalaststand'} \mathcal{I}_g^f(v)=(\mathcal{I}_{u_1}^{u_2} \circ \widetilde{\Phi})(v)=(\mathcal{H}_{u_1}^{u_2} \circ \widetilde{\Phi}) (v). \end{equation} 
\noindent On the other hand, by  proposition \ref{Hironaka function}, we have  the following commutative diagram
$$\xymatrix{
  \ar[rd]^{\mathcal{H}_g^f }  \mathrm{NL}(X,0)\ar[r]^{\widetilde{\Phi}} &\mathrm{NL}(\mathbb{C}^2,0) \ar[d]^{\mathcal{H}_{u_1}^{u_2}} \\
  & \mathbb{R}_{>0} \cup \{+ \infty \},
}$$ which implies that $(\mathcal{H}_{u_1}^{u_2} \circ \widetilde{\Phi}) (v)=\mathcal{H}_{g}^{f}(v)$. Finally, by equation \eqref{hironakalaststand'},  it follow that $$ \mathcal{I}_g^f(v)=\mathcal{H}_{g}^{f}(v).$$  
\end{proof}

\begin{exam}
This example illustrates the statement of Theorem \ref{touten1} in the case $(X,0)=(\mathbb{C}^2,0)$ and $\Phi(x,y)=(g(x,y),f(x,y))=(x+y,y^5-x^{12})$.
  \noindent The following graph is the dual graph of the minimal good resolution $\pi$ of the curve $(gf)^{-1}(0)$   \begin{figure}[H]
\begin{tikzpicture}[node distance=3.8cm, very thick]
 \tikzstyle{titleVertex}      = [ shape=circle,node distance=4cm]
  \tikzstyle{inVertex}      = [ shape=circle,node distance=2cm]
\tikzstyle{Vertex}      = [fill, shape=circle, line cap=round,line join=round,>=triangle 45,scale=.4,font=\scriptsize]
  \tikzstyle{Edge}        = [black]
    \tikzstyle{arrowEdge}        = [black, ->]  
 \tikzstyle{arrowEdge'}        = [black, -<]

  \node[Vertex]      (1)               {};

     \node[Vertex]      (2)  [right       of=1] {};
  \node[Vertex]      (3)  [ right of=2]  {};
  \node[Vertex]      (4)  [right  of=3]  {};
  \node[Vertex]      (5)  [below        of=4] {};
  \node[Vertex]      (6)  [below       of=5] {};
     \node[inVertex] (0)  [left     of=1]  {$(1)$};
                       \node[inVertex] (7)  [ right     of=4]  {$(1)$};

          \path 
            (1) edge[Edge] (2)
      (2) edge[Edge] (3)
             (3) edge[Edge] (4)
               (4) edge[Edge] (5)
                (1) edge[arrowEdge]  (0)
                               (4) edge[arrowEdge']  (7)
                                                                                             (5) edge[Edge] (6);
\draw (1) node[below] {$\frac{1}{5}$}
          (2) node[below] {$\frac{1}{10}$}
          (3) node [below]{$\frac{3}{35}$}
           (4) node [below left ]{$\frac{1}{12}$}
            (5) node [right]{$$}
             (6) node [right]{$$}
              (1) node[above] {$(\frac{1}{5})$}
          (2) node[above] {$(\frac{1}{10})$}
          (3) node [above]{$(\frac{3}{35})$}
           (4) node [above]{$(\frac{1}{12})$}
            (5) node [right]{$(\frac{1}{12})$}
             (6) node [right]{$(\frac{1}{12})$}
                 (5) node [left]{$\frac{1}{8}$}
             (6) node [left]{$\frac{1}{4}$};

  \end{tikzpicture}
  \caption{The graph $\Gamma_{\pi}$ is weighted with the inner rates (without parenthesis) and the Hironaka quotients (between parenthesis). }\end{figure}
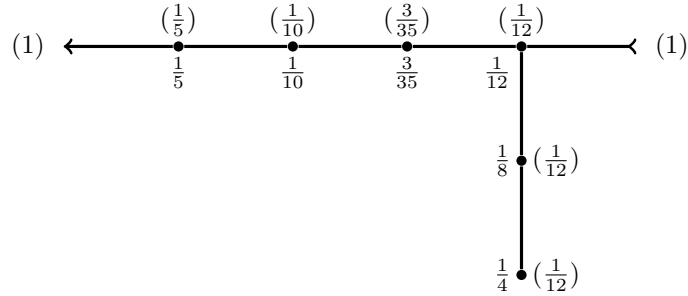

\end{exam}

 We now recall the growth behaviour of the Hironaka quotients function which was proved  by Maugendre and Michel in \cite{MaugendreMichel2017}.

\begin{thm}[{\cite[Theorem 1]{MaugendreMichel2017}}] \label{hironakagrowth}
We denote by $\mathcal{G}_{\pi}$ the union of every path  joining an $f$-node to a $g$-node along which the Hironaka quotients function $\mathcal{H}_{g}^{f}$ is strictly increasing. Let $(X,0)$ be a complex surface germ with an isolated singularity and $\Phi=(g,f):(X,0) \longrightarrow (\mathbb{C}^2,0)$ be a finite morphism. Let $\pi:(X_{\pi},E) \longrightarrow (X,0)$ be a good resolution of $(X,0)$. The Hironaka quotients function $\mathcal{H}_g^f$ restricted to $\Gamma_{\pi}$ is constant on every connected component of the topological closure of $\Gamma_{\pi} \backslash \mathcal{G}_{\pi}$.
\end{thm}
\begin{proof}[Proof of  Point $(iii)$ of Theorem \ref{touten1}] By Theorem \cite{MaugendreMichel2017} it suffices to prove that $\mathcal{A}_{\pi} = \mathcal{G}_{\pi}$. Theorem \ref{touten1}$(ii)$ implies directly the inclusion $\mathcal{A}_{\pi} \subset \mathcal{G}_{\pi}$.  Then, we need to prove the inclusion $\mathcal{G}_{\pi} \subset \mathcal{A}_{\pi}$. For that aim,  we prove that the functions $\mathcal{I}_g^f$ and $\mathcal{H}_g^f$ coincide on $\mathcal{G}_{\pi}$. In the case where $(X,0)=(\mathbb{C}^2,0)$, we get directly from Lemmas \ref{recurence} and \ref{recurence'} that $\mathcal{I}_{u_1}^{u_2}$ and $\mathcal{H}_{u_1}^{u_2}$ coincide on $\mathcal{G}_{\sigma}$. The rest of the proof is easily adapted from the arguments of the proof of  Theorem \ref{touten1}$(ii)$ by using the two commutative diagrams of Propositions \ref{inner-rate function} and \ref{Hironaka function}.
\end{proof}



\section{A first application of the inner rates formula}\label{section8}

As a first application of the inner rates formula \ref{laplacien}, we give a new proof of \cite[Theorem 4.9]{Michel2008}.

\begin{thm}[{\cite[Theorem 4.9]{Michel2008}}]\label{francoise}  Let $(X,0)$ be a complex surface germ with an isolated singularity and let $g,f:(X,0) \longrightarrow (\mathbb{C},0)$ be two holomorphic functions on $X$ such that the morphism $\Phi=(g,f): (X,0) \longrightarrow (\mathbb{C}^2,0)$ is finite. Let $\pi :(X_{\pi},E) \longrightarrow (X,0) $ be a good resolution of $(X,0)$. Let $\mathcal{Z}$ be  a connected component of $\overline{\Gamma_{\pi} \backslash \mathcal{A}_{\pi}}$ or a single vertex on the complementary of  $\overline{\Gamma_{\pi} \backslash \mathcal{A}_{\pi}}$, then :


$$  \sum_{v \in \mathcal{Z} } m_v(f) \Pi_\Phi^* \cdot   E_v=- \left(  \sum_{v \in \mathcal{Z}} m_v(f) \chi'_v  \right).$$ where $\chi'_v:=2-2g_v-\mathrm{val}(v) -f^* \cdot E_v-g^* \cdot E_v$

\end{thm} \begin{rema} The theorem \ref{francoise} is more general than the original statement  \cite[Theorem 4.9]{Michel2008} because we do not assume that $\pi$ is a good resolution of the curve $(gf)^{-1}(0)$. Notice that without that hypothesis the number $ \chi'_v$ is not, in general, equal to the euler characteristic of the surface $ E'_v:=E_v-\left( \left(\bigcup_{i \neq v} E_i \right)\cup f^* \cup g^* \right)$ which appears in the original statement.   \end{rema}


\begin{proof}[Proof of Theorem \ref{francoise}]
Let us recall the inner rates formula \ref{laplacien} applied on  an irreducible component  $E_v$ of the exceptional divisor

\begin{eqnarray*}  \left( \sum_{i \in V(\Gamma_{\pi})} m_{i}(f)q_{g,i}^f E_i  \right) \cdot E_{v}=\mathrm{val}(v) +f^*\cdot E_v-\Pi_{\Phi}^* \cdot E_v+2g_v-2.
\end{eqnarray*}  \label{inner1}
\noindent Let us add  the term $g^* \cdot E_v$ to both sides of this  equality knowing that $$g^* \cdot E_v=-\left(\sum_{i \in V(\Gamma_{\pi})} m_i(g) E_i \cdot E_v \right) \  (proposition \ \ref{laufer}).$$ We get 

$$  \sum_{i \in V(\Gamma_{\pi})} \left( m_{i}(f)q_{g,i}^f -m_i(g) \right) E_i \cdot E_v  =\mathrm{val}(v) +f^* \cdot E_v+g^* \cdot E_v-\Pi_{\Phi}^* \cdot E_v+2g_v-2,
$$ 
by definition of  $\chi'_v$, the last equation becomes 

$$  \sum_{i \in V(\Gamma_{\pi})} \left( m_{i}(f)q_{g,i}^f -m_i(g) \right) E_i \cdot E_v  =-\Pi_{\Phi}^* \cdot E_v-\chi'_v. $$ 
Let $\mathcal{Z}$ be as in the statement of the theorem. Let us multiply both side of the last equality by $m_v(f)$ and take the sum over $v $ in $V(\mathcal{Z})$

$$\sum_{v \in V(\mathcal{Z})}  \sum_{i \in V(\Gamma_{\pi})} \left(m_v(f) m_{i}(f)q_{g,i}^f -m_v(f)m_i(g) \right) E_i \cdot E_v  =-\sum_{v \in V(\mathcal{Z})}  m_v(f)\Pi_{\Phi}^* \cdot E_v+m_v(f)\chi'_v.$$ It remains to prove that the left side $$A=\sum_{v \in V(\mathcal{Z})}  \sum_{i \in V(\Gamma_{\pi})} \left(m_v(f) m_{i}(f)q_{g,i}^f -m_v(f)m_i(g) \right) E_i \cdot E_v $$ is equal zero. Let us inject the equality $m_v(f) E_{v}^2= - \sum_{i \neq v} m_i(f) E_i \cdot E_v +f^* \cdot E_v$ (by proposition \ref{laufer}) in A : 
$$A= \sum_{v \in V(\mathcal{Z})}  \sum_{i \in V(\Gamma_{\pi})\\ i\neq v} \left(m_i(f)m_v(f) (q_{g,i}^f-q_{g,v}^f)+m_i(f)m_v(g)-m_i(g)m_v(f) \right) E_i \cdot E_v.$$

\noindent Let us notice that for every vertex $v$  of $\mathcal{A}_{\pi} \cap \mathcal{Z}$ we have:

$$ \sum_{i \in V({\mathcal{A}_{\pi}})\\ i\neq v} \left(m_i(f)m_v(f) (q_{g,i}^f-q_{g,v}^f)+m_i(f)m_v(g)-m_i(g)m_v(f) \right) E_i \cdot E_v =0,$$
because $q_{g,i}^f =\frac{m_i(g)}{m_i(f)}$ for every vertex $i$ of $\mathcal{A}_{\pi}$. On the other hand, if $v$ is a vertex of $\mathcal{Z} \backslash \mathcal{A}_{\pi}$, all it's neighbor vertices are in $\mathcal{Z}$, this imply that
  $$A= \sum_{v \in V(\mathcal{Z})}  \sum_{i \in V({\mathcal{Z}}),\\ i\neq v} \left(m_i(f)m_v(f) (q_{g,i}^f-q_{g,v}^f)+m_i(f)m_v(g)-m_i(g)m_v(f) \right) E_i \cdot E_v . $$
  By Theorem \ref{touten1}$(ii)$, $\frac{m_i(g)}{m_i(f)}=\frac{m_j(g)}{m_j(f)}$, for every $i,j \in V(\mathcal{Z})$. Therefore 
   $$A= \sum_{v \in V(\mathcal{Z})}  \sum_{i \in V({\mathcal{Z}})\\ i\neq v} m_i(f)m_v(f) (q_{g,i}^f-q_{g,v}^f) E_i \cdot E_v . $$
  By reordering the terms of $A$, we obtain:
  $$ A=\sum_{v,i \in V(\mathcal{Z}), v \neq i}(q_{g,i}^f-q_{g,v}^f)(m_i(f)m_v(g)-m_v(f)m_i(g)) E_i \cdot E_v.$$ Again, by Theorem \ref{touten1}$(ii)$, all the terms of $A$ are equal to zero therefore $A=0$ as desired. \end{proof}

\section{Polar exploration}\label{section9}

As an application of the inner rates formula (Theorem \ref{laplacien}),  we will perform \textit{polar exploration} using the inner rates expanding the ideas of \cite{BFP,BFNP,BFP2}. The main aim of this section is to prove Proposition \ref{propositionA} .
\begin{defi}\label{embeddedtopologicaltype}
Let $(X_1,0)$ and $(X_2,0)$ be two complex analytic surface germ with an isolated singularity. Let $(C_1,0)$ and $(C_2,0)$ be two germ of complex analytic curves embedded in $(X_1,0)$ and $(X_2,0)$ respectively. We say that the curves $(C_1,0)$ and $(C_2,0)$ have the \textbf{same  embedded topological type} if there exists a germ of homeomorphism $\psi: ( X_1,C_1 ) \longrightarrow (X_2, C_2 )$.
\end{defi} \noindent Neumann gave in \cite{neu81}  a complete invariant of the embedded topological type of a germ of  complex analytic curve embedded in a complex surface with an isolated singularity as follows. Let $(C,0)$ be a germ of analytic curve embedded in a germ of complex analytic surface $(X,0)$.  The embedded topological type  of the curve $(C,0)$ determines and is determined by the dual graph of the minimal good resolution $\pi$ of $(X,0)$  and $(C,0)$ decorated with arrows corresponding to the irreducible components of the strict transform of $C$ by $\pi$.\\

Let $(X,0)$ be a complex analytic surface germ with an isolated singularity and let $\Phi=(g,f) : (X,0) \longrightarrow (\mathbb{C}^2,0)$ be a finite morphism. Let $\pi :(X_{\pi},E) \longrightarrow (X,0)$ be a good resolution of $(X,0)$ and $(gf)^{-1}(0)$.  Let $E_{v_1},E_{v_2},\ldots,E_{v_n}$ be the irreducible components of $E$.   Let us recall the inner rates formulas applied on morphisms $(g,f)$ and $(f,g)$:

$$M_{\pi}  \cdot \underline{a_{g,\pi}^f}=\underline{K_{\pi}}+\underline{F_{\pi}}-\underline{P_{\pi}}, \ \ \  M_{\pi} \cdot  \underline{a_{f,\pi}^g}=\underline{K_{\pi}}+\underline{G_{\pi}}-\underline{P_{\pi}}$$
where $M_{\pi}=(E_{v_i} \cdot E_{v_j})_{i,j \in \{1,2,\ldots,n\}}$, $a_{g,\pi}^f:=(m_{v_1}{(f)}q_{g, v_1}^f,\ldots,m_{v_n}(f)q_{g, v_n}^f)$, $K_{\pi} :=( \mathrm{val}_{\Gamma_{\pi}} (v_1) +2g_{v_1}-2,\ldots,\mathrm{val}_{\Gamma_{\pi}} (v_n) +2g_{v_n}-2)$, $F_{\pi}=(f^* \cdot E_{v_1},\ldots,f^* \cdot E_{v_n} ) $, $G_{\pi}= (g^* \cdot E_{v_1}, \ldots , g^* \cdot E_{v_n})$,  and 
$P_{\pi}=(\Pi_{\Phi}^* \cdot E_{v_1},\ldots,\Pi_{\Phi}^* \cdot E_{v_n})$. What we call \textbf{polar exploration} consists of finding the $\mathcal{P}$-vector  $P_{\pi}=(\Pi_{\Phi}^* \cdot E_{v_1},\ldots,\Pi_{\Phi}^* \cdot E_{v_n}) $ knowing the vectors $M_{\pi}$, $K_{\pi}$, $F_{\pi}$ and $G_{\pi}$ in the case $\pi$ is the minimal good resolution of $(X,0)$ and $(gf)^{-1}(0)$. In other words, we wants to find all the possible $\mathcal{P}$-vector for a fixed  embedded topological type of the curve $(gf)^{-1}(0)$.

 Let us give  an example where for a given topological type the polar exploration gives three possible $\mathcal{P}$-vectors. We prove they are all realized by a morphism.  \begin{exam}\label{patricio}
 Consider again Example \ref{exemple1}.\begin{figure}[H]
\begin{tikzpicture}[node distance=4.5cm, very thick]
 \tikzstyle{titleVertex}      = [ shape=circle,node distance=4cm] \tikzstyle{inVertex}      = [ shape=circle,node distance=2.5cm]
\tikzstyle{Vertex}      = [fill, shape=circle, line cap=round,line join=round,>=triangle 45,scale=.4,font=\scriptsize]
  \tikzstyle{Edge}        = [black]
    \tikzstyle{arrowEdge}        = [black, ->]  
   \tikzstyle{arrowEdge'}        = [black, -<]    
     \tikzstyle{arrowEdge''}        = [red, ->]

 \node[Vertex]      (1)               {};

     \node[Vertex]      (2)  [right       of=1] {};
  \node[Vertex]      (3)  [ right of=2]  {};
  \node[Vertex]      (4)  [right  of=3]  {};
  \node[Vertex]      (5)  [below        of=4] {};
  \node[Vertex]      (6)  [below       of=5] {};
    \node[inVertex] (0)  [ left     of=1]  {$(1)$};    
      \node[inVertex] (7)  [ right     of=4]  {$(1)$};   

          \path 
(1) edge[Edge] (2)
(2) edge[Edge] (3)
(3) edge[Edge] (4)
(4) edge[Edge] (5)
(4) edge[arrowEdge']  (7)
(1) edge[arrowEdge]  (0)
(5) edge[Edge] (6);

\draw (1) node[below] {$-2$}
          (2) node[below] {$-4$}
          (3) node [below]{$-2$}
           (4) node [below left ]{$-1$}
            (5) node [right]{$$}
             (6) node [right]{$$}
              (1) node[above] {$(1,5)$}
          (2) node[above] {$(1,10)$}
          (3) node [above]{$(3,35)$}
           (4) node [above]{$(5,60)$}
            (5) node [right]{$(2,24)$}
             (6) node [right]{$(1,12)$}
                 (5) node [left]{$-3$}
             (6) node [left]{$-2$};

                                      
  \end{tikzpicture} 
  \caption{The numbers between parenthesis are the orders of vanishing $(m_v(g),m_v(f))$ of  the functions $g \circ \pi $ and $f \circ \pi$ along the irreducible components of $E$, these numbers can be determined from the dual graph using Proposition \ref{laufer}.}
   \end{figure}
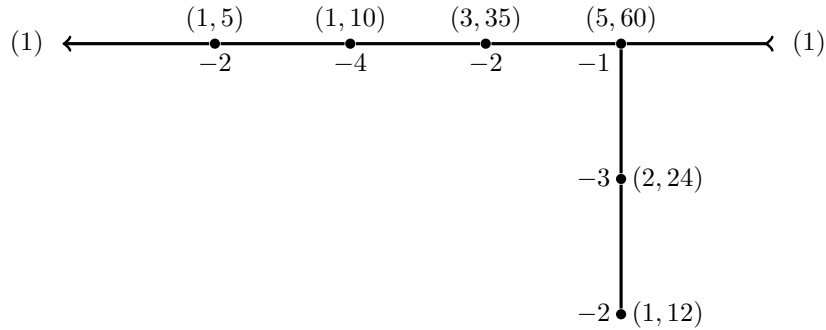

\noindent Now, we will focus on determining the $\mathcal{P}$-vector of $\Phi$ using only the dual graph $\Gamma_{\pi}$.  By applying \cite[Theorem 4.9]{Michel2008}(Theorem \ref{francoise})  successively  to each vertex   $v_1,v_2$ and $v_3$ of $\mathcal{A}_{\pi}$, we get $\Pi_{\Phi}^* \cdot E_{v_i}=0$ for $i=1,2,3$. Now we apply the theorem on the set $\{ v_4,v_5,v_6\}$ which is a connected component of $\overline{\Gamma_{\pi} \backslash \mathcal{A}_{\pi}}$ and we get the equation
$$ 60\Pi_{\Phi}^* \cdot E_{v_4} +24\Pi_{\Phi}^* \cdot E_{v_5} +12\Pi_{\Phi}^* \cdot E_{v_6} = 48.$$ 
 this gives  three possible $\mathcal{P}$-vectors $$( \Pi_{\Phi}^* \cdot E_{v_{1}},\Pi_{\Phi}^* \cdot E_{v_{2}},\Pi_{\Phi}^* \cdot E_{v_{3}},\Pi_{\Phi}^* \cdot E_{v_{4}},\Pi_{\Phi}^*\cdot E_{v_{5}},\Pi_{\Phi}^* \cdot E_{v_{6}}) \in \left \{
\begin{array}{rcl}
 (0,0,2,1,0,0),&& \\
 (0,0,4,0,0,0),&&\\
(0,0,0,2,0,0)&&
\end{array}\right\}$$

 We already know  from the computations made in Example \ref{exemple3.5} that the $\mathcal{P}$-vector  $(0,0,2,1,0,0)$ corresponds to the morphism $\Phi$. Now, a natural question is wether   the two vectors $(0,0,4,0,0,0),(0,0,0,2,0,0)$ can be realized by the $\mathcal{P}$-vectors of two other morphisms. The answer is yes, we will show it by using the following lemma:
\begin{lemm}[{\cite[lemma 6.6.1]{casas}}] \label{casas}
Let $\gamma$ be a germ of complex curve at the origin of  $\mathbb{C}^2$. The following   assertions are equivalent:
\begin{enumerate}
\item $\gamma$ has equation $$ h(x,y)=a_{0,n}y^n+a_{m,0}x^m+ \sum_{ni+mj>nm}a_{i,j}x^iy^j=0$$
with $a_{0,n} \neq 0$ and $a_{m,0} \neq 0$.
\item $\gamma$ has the same embedded topological type as the curve of equation
$$ y^n+x^m=0$$ 
\end{enumerate}
\end{lemm}
By lemma \ref{casas}  the  dual graph of the minimal good resolution of the curves  $(x,y^5-x^{12})=0$ and $(x+y)(y^5-x^{12}+x^{10}y)=0$  is $\Gamma_{\pi}$. A direct computation shows that:
\begin{enumerate}

\item The polar curve ${\Pi_{\Phi_1}}$ of the morphism $\Phi_1=(x,y^5-x^{12})$ has equation $ y^4=0$, the $\mathcal{P}$-vector  is 
$$  (0,0,4,0,0,0).$$

\item The polar curve ${\Pi_{\Phi_2}}$ of the morphism $\Phi_2=(x+y,y^5-x^{12}+x^{10}y)$ has equation $ 5y^4+x^{10}+12x^{11}-10x^9y=0$, the $\mathcal{P}$-vector  is 
$$ (0,0,0,2,0,0)$$
\end{enumerate}

  \end{exam}  In Example \ref{patricio}, Theorem \ref{francoise} (\cite[Theorem 4.9]{Michel2008}) is sufficient to find all the realized $\mathcal{P}$-vectors. We will now prove Proposition \ref{propositionA} stated in the introduction which provides a family of  examples where the knowledge of the inner rates and their properties (Theorem \ref{touten1}) gives a better restriction on the possible  $\mathcal{P}$-vectors than Theorem \ref{francoise}. In fact  we will prove the following  more precise version of Proposition \ref{propositionA}.

  \begin{prop}\label{famille}
  Let $n \in \mathbb{N}_{\geq 2}$ and consider  the following graph with arrows  $\Gamma_n$:
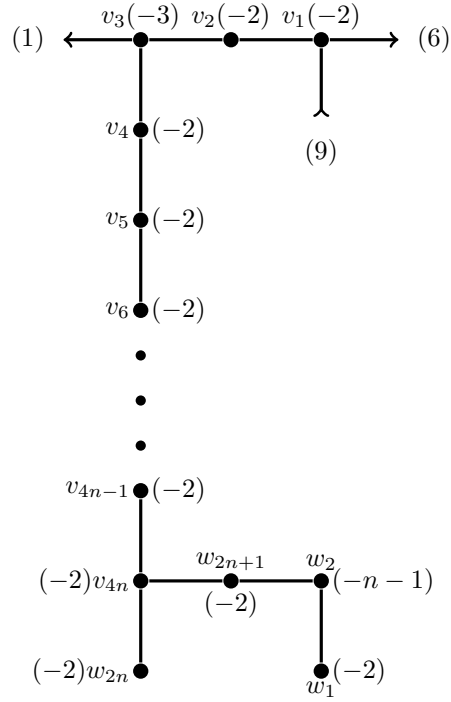
\begin{figure}[H]
\begin{tikzpicture}[node distance=2cm, very thick]
 \tikzstyle{titleVertex}      = [ shape=circle,node distance=4cm]
 \tikzstyle{pttVertex}      = [fill, shape=circle, line cap=round,line join=round,>=triangle 45,scale=.4,font=\scriptsize,node distance=1.5cm]

  \tikzstyle{inVertex}      = [ shape=circle,node distance=1.5cm]
\tikzstyle{Vertex}      = [fill, shape=circle, line cap=round,line join=round,>=triangle 45,scale=.6,font=\scriptsize]
  \tikzstyle{Edge}        = [black]
    \tikzstyle{arrowEdge}        = [black, ->]  
    \tikzstyle{arrowEdge'}        = [black, -<]

 \node[Vertex]      (1)               {};
   \node[inVertex] (a)  [right       of=1]  {$(6)$};
      \node[inVertex] (b)  [below       of=1]  {$(9)$};
           \node[Vertex]      (2)  [left       of=1] {};
  \node[Vertex]      (3)  [ left of=2]  {};
  \node[inVertex] (c)  [left    of=3]  {$(1)$};   
  \node[Vertex]      (4)  [below of=3]  {};
  \node[Vertex]      (5)  [below of=4] {};
  \node[Vertex]      (6)  [below of=5] {};

     \node[pttVertex]      (7)  [below of=6] {};
     \node[pttVertex]      (8)  [below of=7] {};   
        \node[pttVertex]      (8')  [below of=8] {};   
        
                \node[Vertex]      (9)  [below of=8] {};
          \node[Vertex]      (10)  [below of=9] {};        
         \node[Vertex]      (11)  [below of=10] {};
                     \node[Vertex]      (12)  [right of=10] {};   
                     \node[Vertex]      (13)  [right of=12] {};   
\node[Vertex]      (14)  [below of=13] {};

                  \path (1) edge[arrowEdge]  (a)
                  (1) edge[arrowEdge']  (b)
                  (3) edge[arrowEdge]  (c)
                                                (1) edge[Edge] (2)
      (2) edge[Edge] (3)
             (3) edge[Edge] (4)
               (4) edge[Edge] (5)
         (5) edge[Edge] (6)
                        (9) edge[Edge] (10)
               (10) edge[Edge] (11)
               (10) edge[Edge] (12)
               (12) edge[Edge] (13)                              
               (13) edge[Edge] (14);                
               \draw (1) node[above] {$v_1(-2)$}
          (2) node[above] {$v_2(-2)$}
          (3) node [above]{$v_3(-3)$}
  
                                          (4) node [left]{$v_4$}
            (5) node [left]{$v_5$}
             (6) node [left]{$v_6$}
              (9) node [left]{$v_{4n-1}$}
                               (10) node [left]{$(-2)v_{4n}$}
                               (11) node [left]{$(-2)w_{2n}$}
                               (12) node [above]{$w_{2n+1}$}
                               (13) node [above]{$w_{2}$}
                               (14) node [below]{$w_{1}$}
                               (4) node [right]{$(-2)$} 
                               (5) node [right]{$(-2)$}
                               (6) node [right]{$(-2)$}
                               (9) node [right]{$(-2)$}

 (12) node [below]{$(-2)$}
  (13) node [right]{$(-n-1)$}
 (14) node [right]{$(-2)$};
                                                                 \end{tikzpicture}
                                                                 \caption{The genus is $0$ for every vertex}
                                                                 
                                                                  \end{figure}                                                              
\begin{enumerate}
\item There exists a  complex surface singularity  $(X_n,0)$ and a finite morphism $\Phi_n=(g_n,f_n) :( X_n,0) \longrightarrow (\mathbb{C}^2,0)$ such that   $\Gamma_n$ is the dual graph of the minimal good resolution of $(X_n,0)$ and $(g_nf_n)^{-1}(0)$;

\item The $\mathcal{P}$-vector of any such morphism $\Phi_n$ belongs to a set of $n+5$ elements. More precisely:\\ $\Pi_{\Phi_n}^* \cdot E_{v_1}=14$  and $\Pi_{\Phi_n}^* \cdot E_{v_k}=0$ for every $k = 2,3,\ldots,4n, $ \\ $( \Pi_{\Phi_n}^* \cdot E_{w_{2n}},\Pi_{\Phi_n}^* \cdot E_{w_{2n+1}},\Pi_{\Phi_n}^* \cdot E_{w_{2}},\Pi_{\Phi_n}^* \cdot E_{w_{1}}) \in \left \{
\begin{array}{rcl}
 (1,0,1,0),&& \\

 (1,0,0,2),&&\\

(0,1,0,1),&&\\

(0,0,k,k+2),&\\ 0 \leq k \leq n+1&
\end{array}
\right\}$
\end{enumerate}

\end{prop}

   In order to prove Proposition \ref{famille}  we will need some tools we are going to present now.
 All the definitions and results  of this subsection can be found e.g in \cite{nem}.
 
  We call \textbf{cycle} on $\Gamma_{\pi}$ every effective divisor of $X_{\pi}$ which is supported on $E$. Let $h :(X,0) \longrightarrow (\mathbb{C},0) $ be a holomorphic function. Denote by $(h_\pi)$ the compact part of the total transform oh $h$ by $\pi$ i.e,$$ (h_\pi)=\sum_{i \in V(\Gamma_{\pi})} m_i(h)E_i.$$ The cycle $(h_\pi)$ is called analytic cycle and we denote by $Z_{an}(\Gamma_{\pi} )$ the set of \textbf{analytic cycles} on $\Gamma_{\pi}$. A  \textbf{topological cycle} $D$  is an effective  cycle such that $D \cdot E_v \leq 0 $ for every irreducible component $E_v$ of $E$, we denote by $Z_{top}(\Gamma_{\pi})$ the set of such cycles.

There is a natural ordering of the cycles: $D_1 = \sum n_i E_{v_i} \leq D_2 = \sum n'_i E_{v_i}$ if and only if $n_i \leq n_i'$ for all $i$. By \cite[Lemma 2.6]{nem} the set $Z_{top}(\Gamma_{\pi})$ contains a minimal element denoted $Z_{min}(\Gamma_{\pi})$ called the \textbf{minimal cycle}.
\begin{defi}[{\cite[Appendix 2]{nem}}]
The \textbf{Euler characteristic of an effective cycle} $D$ is defined by  
$$ \chi(D) =-\frac{1}{2}D\cdot(D+K_{\pi})$$
where $K_{X_\pi}$ is the canonical divisor of $X_{\pi}$.
\end{defi}

\begin{defi}[{\cite{nem}[Theorem 3.8]}, \cite{artin}]
The complex surface germ $(X,0)$ is said to be \textbf{rational} if $\chi(Z_{min}(\Gamma_{\pi}))=1$ where $\pi$ is a good resolution of $(X,0)$.
\end{defi}
\begin{thm}[{\cite[Theorem 3.14]{nem},  \cite{artin} }] \label{artin} Let $(X,0)$ be a complex surface germ with an isolated singularity. If $(X,0)$ is rational then  $$  Z_{top}(\Gamma_{\pi} ) = Z_{an}(\Gamma_{\pi} ),$$ for every good resolution $\pi$ of $(X,0)$.
 \end{thm}

Now, we are ready to prove Proposition \ref{famille}.

\begin{proof}[Proof of Proposition \ref{famille}]

Let us first prove the existence of the surfaces $(X_n,0)$. By a classical result of  Grauert  (\cite{Greu})   every weighted graph  without loops and with negative definite intersection matrix can be realized as the dual graph $\Gamma_{\pi}$ associated with the minimal good resolution of some complex analytic surface germ $(X, 0)$ (See e.g \cite[Section 2]{NP2007} for details). The intersection matrix of the dual graph $\Gamma_{n}$ is negative definite, therefore, there exists a  germ complex surface singularity $(X_n,0)$ such that the dual graph of its minimal good resolution $\pi_n:(X_{\pi_n},E_n) \longrightarrow (X_n,0)$ is $\Gamma_{n}$.\\

Now, we  prove the existence of the morphisms $\Phi_n=(g_n,f_n).$ The minimal cycle  of $\Gamma_{\pi_n}$  denoted  $Z_{min}(\Gamma_{\pi_n})$  is obtained by using  \textit{Laufer's algorithm} \cite{laufer1972} (see e.g,  \cite[2.10]{nem} for more details):
                         $$ Z_{min}(\Gamma_{\pi_n})= E_{v_1}+E_{v_2}+E_{v_3}+2(E_{v_4}+E_{v_5}+\ldots+E_{v_{4n-1}}+E_{v_{4n}}+E_{w_{2n+1}})+E_{w_2}+E_{w_1}+E_{w_{2n}}.$$
 A direct computation using the adjunction formula \ref{adj} shows that $\chi(Z_{min}(\Gamma_{\pi_n}))=1$. 
 By using Lemma \ref{laufer} we now compute from the dual graph $\Gamma_{\pi_n}$ the orders of vanishing of the functions $f_n \circ \pi_n $ and $g_n \circ \pi_n $ along the irreducible components of $E_n$:
 
\begin{figure}[H]
  \begin{tikzpicture}[node distance=2cm, very thick]

 \tikzstyle{titleVertex}      = [ shape=circle,node distance=4cm]
 \tikzstyle{pttVertex}      = [fill, shape=circle, line cap=round,line join=round,>=triangle 45,scale=.4,font=\scriptsize,node distance=1.5cm]

  \tikzstyle{inVertex}      = [ shape=circle,node distance=1.5cm]
\tikzstyle{Vertex}      = [fill, shape=circle, line cap=round,line join=round,>=triangle 45,scale=.6,font=\scriptsize]
  \tikzstyle{Edge}        = [black]
    \tikzstyle{arrowEdge}        = [black, ->]  
        \tikzstyle{arrowEdge'}        = [black, -<]

 \node[Vertex]      (1)               {};
   \node[inVertex] (a)  [right       of=1]  {$(6)$};
      \node[inVertex] (b)  [below       of=1]  {$(9)$};
           \node[Vertex]      (2)  [left       of=1] {};
  \node[Vertex]      (3)  [ left of=2]  {};
  \node[inVertex] (c)  [left   of=3]  {$(1)$};   
  \node[Vertex]      (4)  [below of=3]  {};
  \node[Vertex]      (5)  [below of=4] {};
  \node[Vertex]      (6)  [below of=5] {};

     \node[pttVertex]      (7)  [below of=6] {};
     \node[pttVertex]      (8)  [below of=7] {};   
        \node[pttVertex]      (8')  [below of=8] {};   
        
                \node[Vertex]      (9)  [below of=8] {};
          \node[Vertex]      (10)  [below of=9] {};        
         \node[Vertex]      (11)  [below of=10] {};
                     \node[Vertex]      (12)  [right of=10] {};   
                     \node[Vertex]      (13)  [right of=12] {};   
\node[Vertex]      (14)  [below of=13] {};

                  \path (1) edge[arrowEdge]  (a)
                  (1) edge[arrowEdge']  (b)
                  (3) edge[arrowEdge]  (c)
                                                (1) edge[Edge] (2)
      (2) edge[Edge] (3)
             (3) edge[Edge] (4)
               (4) edge[Edge] (5)
       (5) edge[Edge] (6)
                      (9) edge[Edge] (10)
               (10) edge[Edge] (11)
               (10) edge[Edge] (12)
               (12) edge[Edge] (13)                              
               (13) edge[Edge] (14);                
               \draw (1) node[above] {$v_1(5,7)$}
          (2) node[above] {$v_2(4,5)$}
          (3) node [above]{$v_3(3,3)$}
  
                                          (4) node [left]{$v_4$}
            (5) node [left]{$v_5$}
             (6) node [left]{$v_6$}
              (9) node [left]{$v_{4n-1}$}
                               (10) node [left]{$(4n,4n)v_{4n}$}
                               (11) node [left]{$(2n,2n)w_{2n}$}
                               (12) node [above]{$w_{2n+1}$}
                               (13) node [above]{$w_{2}$}
                               (14) node [below]{$w_{1}$}
                               (4) node [right]{$(4,4)$} 
                               (5) node [right]{$(5,5)$}
                               (6) node [right]{$(6,6)$}
                               (9) node [right]{$(4n-1,4n-1)$}
                               (12) node [below]{$(2n+1,2n+1)$}
                               (13) node [right]{$(2,2)$}
                               (14) node [right]{$(1,1)$};
                                                                 \end{tikzpicture} \caption{The numbers between parenthesis are the orders of vanishing $(m_v(g_n),m_v(f_n))$ of  the functions $g_n \circ \pi_n $ and $f_n \circ \pi_n$ along the irreducible components of $E_n$, these numbers can be determined from the dual graph using Proposition \ref{laufer}.}
                                                                 \end{figure}
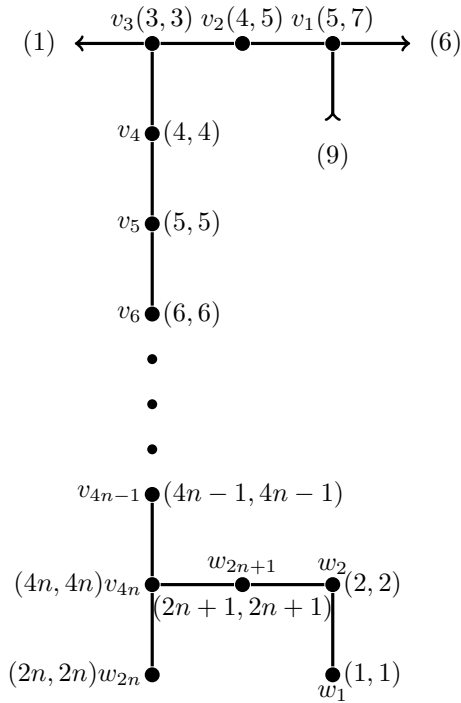
                                                                 
   We are going to construct two holomorphic function germs $g_n,f_n:(X_n,0) \longrightarrow (\mathbb{C},0)$  whose  analytic cycles are respectively:
 $$D_{n,1}= 7E_{v_1}+5E_{v_2}+\sum_{k=3}^{4n} kE_{v_k} +2nE_{w_{2n}}+(2n+1)E_{w_{2n+1}} +2E_{w_2}+E_{w_1}  $$
$$D_{n,2}= 5E_{v_1}+4E_{v_2}+\sum_{k=3}^{4n} kE_{v_k} +2nE_{w_{2n}}+(2n+1)E_{w_{2n+1}} +2E_{w_2}+E_{w_1}.$$ 
 and such that $f^*_n \cup g^*_n$ is a disjoint union of curvettes of $E_n$.

Let  $a_1,a_2,\ldots a_{15}$ be distinct  smooth points of $E_{v_1}$. Let $e_n$ be the composition of the blowups of the   points blowup $a_i$ and let $E_{v'_i}= e_n^{-1}(a_i)$  for $i=1,\ldots 15$. Since $\chi(Z_{min}(\Gamma_{\pi_n}))=1$, by Theorem \ref{artin}, there exists  functions $f_n$ and $g_n$ whose analytic cycles on $\Gamma_{e_n \circ \pi_n}$ are $$D'_{n,1}=D_{n,1}+ 7(E_{v'_1}+\ldots E_{v'_{6}})+8(E_{v'_7}+\ldots+ E_{v'_{15} } )$$
$$D'_{n,2}=D_{n,1}+ 8(E_{v'_1}+\ldots E_{v'_{7}})+7(E_{v'_1}+\ldots E_{v'_{15}}).$$
Using Lemma \ref{laufer} we can check that $f_n^* \cup g_n^*$ is a disjoint union of curvettes of the exceptional divisor of $e_n \circ \pi_n$ such that $f_n^*\cdot E_{v'_i}=0$ and $g_n^*\cdot E_{v'_i}=1$  for $i=1,\ldots, 6$, and $f_n^*\cdot E_{v'_i}=1$ and $g_n^*\cdot E_{v'_i}=0$  for $i=7,\ldots, 15$.  Therefore blowing down the fifteen components $E_{v'_i}$ we obtain that the union of strict transforms of $f_n$ and $g_n$ by $\pi_n$ is a disjoint union of curvettes of $E_n$.


Let us now prove $(ii)$.  Let $\Phi_n$ be a morphism as in $(i)$. Let compute its possible $\mathcal{P}$-vectors. Until the end of the proof we will use the notations $p_i:=\Pi_{\Phi_n}^* \cdot E_{v_i}$ and $ p'_i:=\Pi_{\Phi_n}^* \cdot E_{w_i}$.\\ By Theorem \ref{touten1}$(ii)$, we  know  that $q_{g_n,v_{1} }^{f_n}=\frac{5}{7}$, $q_{g_n, v_{2} }^{f_n}=\frac{4}{5}$ and $q_{g_n, v_{3} }^{f_n}=1$ . Let us apply Theorem \ref{laplacien} on the vertices $v_1$ and $v_2$, we then obtain $p_1=14 $ and $p_2=0$. Apply again Theorem \ref{laplacien} on the vertex $v_3$, we obtain $ p_3=5-4q_{g_n, v_4}^{f_n}$. By Theorem \ref{touten1}$(i)$ we have $q_{g_n, v_4}^{f_n}>q_{g_n,v_3 }^{f_n}=1$, then $p_3=5-4q_{g_n, v_4}^{f_n}<1$  and it  follows $p_3=0,q_{g_n ,v_4}^{f_n}=\frac{5}{4}$. \\ We now prove by induction that for every $ k \in \{ 3,4,\ldots,4n\}$ we have $q_{g_n, v_{k} }^{f_n} =\frac{2k-3}{k}$ and  $ p_{k-1}=0 $  . It is true for $k=3$. Let us suppose it is true for  a given $k < 4n$. By applying the inner rates formula on the vertex $E_k$ we get $$ p_k= 2k-1-(k+1)q_{g_n ,v_{k+1} }^{f_n}.$$ 
By Theorem \ref{touten1}$(i)$ we know that $q_{g_n ,v_{k+1} }^{f_n}>q_{g_n ,v_{k} }^{f_n}=\frac{2k-3}{k}$. Then
$$p_k=2k-1-(k+1)q_{g_n, v_{k+1} }^{f_n}< 2k-1-\frac{2k-3}{k}(k+1) =\frac{3}{k}<1.$$
Therefore $p_k=0$ and it follows that $q_{g_n, v_{k+1} }^{f_n}= \frac{2k-1}{k+1}$ as desired.

 Let us now  apply  Theorem \ref{francoise} on the vertices $\{v_1,\ldots,v_{4n},w_{2n},w_{2n+1},{w_2},w_{1}\}$, we get the following equality:
$$ 2n+2=4np_{4n}+2np'_{2n}+(2n+1)p'_{2n+1}+2p'_2+p'_1.$$
Solving this last equation  ends the proof.
                                 
                                 \end{proof}
                                 \begin{rema}\label{remaramanujan}
                                 In this last proposition if we did not make use of the inner rates and only restricted ourselves to Theorem  \ref{francoise}  we would have ended  with the following equality
                                 $$3p_3+4p_4+5p_5+\ldots+4np_{4n}+2np'_{2n}+(2n+1)p'_{2n+1}+2p'_2+p'_1=2n+2 $$
                                 where  $p'_i:=\Pi_{\Phi_n}^* \cdot E_{w_i}$ and $p_i:=\Pi_{\Phi_n}^* \cdot E_{v_i}$.\\The cases $p'_{2n+1}=1$ and $p'_{2n}=1$     were already treated  so we can suppose that            $p'_{2n+1}=p'_{2n}=0$ \\
                                 By renaming $x_1=p'_1,x_2=p'_2,x_3=p_3,x_4=p_4$, etc.,        we end up with the equation 
                                 $$ x_1+2x_2+3x_3+\ldots+(2n+2)x_{2n+2}=2n+2, \ \ \ x_i \in \mathbb{N}, $$
                              let $\Pi(n)$ be the cardinal of the set of  solutions of this equation. The sequence $\Pi(n)$ is equivalent to  $$ \Pi(n) \sim \frac{1}{(8n+8) \sqrt{3}} \mathrm{exp} \left(  \pi \sqrt{\frac{4n+4}{3} } \right),$$ as was proved by  Hardy and Ramanujan  (see e.g \cite{hardy}). On the other hand,  by using the inner rates in the proof of Proposition \ref{famille}, we end up with only $n+5$ possible cases. 
                              
                                                     \end{rema}

\section{Geometric interpretation of the inner rates}\label{section10}

Let $(X,0) \subset (\mathbb{C}^n,0)$ be a complex surface germ with an isolated singularity at the origin of $\mathbb{C}^n$. The aim of this section is to give a geometrical interpretation of the inner rates as metric invariants of the Milnor fibers of a non constant holomorphic function $f$. First, we will need to define the notion of generic linear form and generic polar curve with respect to a resolution and a function.
\subsection{Generic linear form and generic polar curve}

 All the results of this subsection are an adaption of  of \cite[Section 3]{BNP} and  \cite[Subsection 2.2]{BFP}.

Let $f:(X,0) \longrightarrow (\mathbb{C},0)$  be a non constant holomorphic function. We call  \textbf{Gauss map relative to} $f$ as defined in \cite[page 364]{Teissier1982} the map :

$$\begin{array}{rcl}
\gamma_f :X \backslash \{0\} &\to& \mathbb{P}(\mathbb{C}^n). \\
p &\mapsto &\mathrm{ker} (\mathrm{d}_p f)
\end{array} $$ 
Let us now consider the following complex surface $$\mathcal{N}_f(X) = \overline{   \mathrm{Graph}(\gamma_f) } = \overline{ \{ (p,d) \in X \backslash \{0\} \times \mathbb{P}(\mathbb{C}^n)   \ | \  \mathrm{d}_pf(d)=0 \}   }   \subset \mathbb{C}^n \times \mathbb{P}(\mathbb{C}^n)$$

\begin{defi}[{\cite[Definition page 367]{Teissier1982}}] \label{nashgauss}
The \textbf{Nash transform} of $X$ relative to $f$ is the modification $\nu_f$ defined by
$$\begin{array}{rcl}
\nu_f :  \mathcal{N}_f(X)  &\to& X\\
(p,d)&\mapsto & p
\end{array} $$ 
\end{defi}
\noindent Let $D \in \mathbb{P}(\mathbb{C}^n)$ and consider its dual linear form $\ell_D:\mathbb{C}^n \longrightarrow \mathbb{C}$.
Let $\Omega_f$ be the Zariski open set of $\mathbb{P}(\mathbb{C}^n)$ such that $\Phi_D = ({\ell_D}_{|X},f) : (X,0) \longrightarrow (\mathbb{C}^2,0)$ is a finite morphism.
We will denote by $\Pi_D$ the polar curve associated with the morphism $\Phi_D$, whenever $D$ is an element of $\Omega_f$.

\begin{defi}[{\cite[Definition 1.1 Chapter III]{Spivakovsky1990}}] \label{basepointdefi}
Let $\alpha:(\tilde{X},E) \longrightarrow (X,0)$ be a modification and let $\Omega$ be an open Zariski set of $\mathbb{P}(\mathbb{C}^n)$. Let $\{ (\gamma_{D},0) \subset (X,0) \}_{D \in \Omega }$ be a family of complex curve germs. We say that a point $p$ of $E$ is a \textbf{base point} of $\alpha$ for the family of  curves $\{  (\gamma_{D},0) \}_{D \in \Omega}$ if there exists an open Zariski set $\Omega_p$ of $\mathbb{P}(\mathbb{C}^n)$ contained in $\Omega$ such that $p \in \gamma_{D}^*$ for all $D$ in $\Omega_p$, where $\gamma_{D}^*$ is the strict transform of  the curve $\gamma_{D}$ by the modification $\alpha$.

\end{defi}
We will use the following properties of the Nash transform. For $(i)$ see \cite[1.1 page 417]{Teissier1982} and  $(ii)$ is a direct consequence of the third point of \cite[Corollary 1.3.2 page 420]{Teissier1982}.
\begin{prop}\label{basepoint}
The Nash transform has the  following properties:
\begin{enumerate}
\item  The map $(\gamma_f \circ \nu_f) : \mathcal{N}_f(X) \backslash \nu_f^{-1}(0) \longrightarrow \mathbb{P}(\mathbb{C}^n)$ extends to a holomorphic map $\tilde{\gamma}_f$ on $\mathcal{N}_f(X)$.
\item The Nash transform has no base point for the family of the polar curves. 

 \end{enumerate}

\end{prop}

\noindent The following definition is an adaptation of the definition of the local bilipchitz constant introduced  in \cite[Section 3]{BNP}.
\begin{defi}\label{LBC}
Let $v_p$ be a unit vector of $T_p X$ such that $d_p f (v_p)=0$. The \textbf{local bilipschitz constant of the morphism} $\Phi_D=({\ell_D}_{|X},f)$ is the map defined by:
$$\begin{array}{rcl}
K_D :X \backslash 0 &\to& \mathbb{R}  \cup \{ \infty \} \\
p &\mapsto& \left\{
    \begin{array}{ll}
        \frac{1}{|| d_p \Phi_D(v_p) ||}& \mbox{if }p \notin  \Pi_D \\
        \infty & \mbox{if }p \in  \Pi_D
    \end{array}
\right.
\end{array} $$ 
\end{defi}

\begin{lemm}\label{lipshitz}
Given any neighborhood $U$ of $\Pi_{D}^* \cap \nu_f^{-1}(B_{\epsilon} \cap X)$  in $\mathcal{N}_f(X) \cap \nu_{f}^{-1}(B_{\epsilon} \cap X)$, the local bilipschitz constant $K_D$ of the morphism $\Phi_{D}$ is bounded on $B_{\epsilon} \cap (X \backslash \nu_f(U))$.
\end{lemm}
\begin{proof}

We have

 \begin{eqnarray*} (K_D \circ \nu_f )(p,d)=K_D(p)=\frac{1}{||(0,\ell_D(v_p)) ||} \end{eqnarray*} Let us consider the map 
$$\begin{array}{rcl}
\kappa_D :\mathbb{P}(\mathbb{C}^n) &\to& \mathbb{R}  \cup \{ \infty \} \\
d &\mapsto& \left\{
    \begin{array}{ll}
        \frac{1}{||(0, \ell_D(\frac{d}{||d||})) ||}& \mbox{if } d \notin  D^{\perp} \\
        \infty & \mbox{if } d \in  D^{\perp}
    \end{array}
\right.
\end{array} $$

By definition of the Nash transform we have  $d_pf(d)=0$ for all $(p,d) \in \mathcal{N}_f(X)$. We then have the equality  \begin{eqnarray}\label{lipshitzien} (K_D \circ \nu_f)(p,d)= (\kappa_D \circ \tilde{\gamma}_f)(p,d), \ \ \forall (p,d) \notin \Pi_{D}^* \end{eqnarray} where $\tilde{\gamma}_f$  is the holomorphic extension of the map $\gamma_f \circ \nu_f$  to $\mathcal{N}_f(X)$ (Remark \ref{basepoint}). \\ The map $\kappa_D \circ \tilde{\gamma}_f$ is continuous and takes finite values outside $\Pi_D^*$. This implies that given any open neighbourhood $U$ of $\Pi_{D}^* \cap \nu_f^{-1}(B_{\epsilon} \cap X)$ in $\mathcal{N}_f(X) \cap \nu_f^{-1}(B_\epsilon \cap X)$ the map  $\kappa_D \circ \tilde{\gamma}_f$ is bounded on the compact set $\nu_f^{-1}(B_{\epsilon} \cap X) \backslash U$.
Then, by Equality (\ref{lipshitzien}), the  local bilipschitz constant $K_D$ of  $\Phi_D$ is bounded on $B_{\epsilon}  \cap (X \backslash \nu_f(U) )$.
\end{proof}
Let $\pi: (X_\pi,E) \longrightarrow (X,0)$ be a good resolution which factors through the Nash transform $\nu_f: \mathcal{N}_f(X) \longrightarrow X$  and the blowup of the maximal ideal $b_0: \mathrm{BL}_0(X) \longrightarrow X$ i.e. there exists two modifications $\mu_{\pi}:(X_\pi,E) \longrightarrow (\mathcal{N}_f(X),\nu_{f}^{-1}(0))$ and $\mu'_{\pi}:(X_\pi,E) \longrightarrow (\mathrm{BL}_0(X),b_0^{-1}(0))$ such that $\nu_f \circ \mu_{\pi} =\pi$ and $b_0 \circ \mu'_{\pi} =\pi$. Let us denote by $E_f$ the strict transform of $\nu_{f}^{-1}(0)$ by the modification $\mu_{\pi}$ and by $E_0$ the strict transform of $b_{0}^{-1}(0)$ by the modification $\mu'_{\pi}$ . By Remark \ref{basepoint} there exists an open Zariski set $\Omega_f^{\pi}$ included in $\Omega_f$ such that for any element $D$ of $\Omega_f^{\pi}$, the strict transforms of the  curves $\Pi_{D}$ and $\ell_{D}^{-1}(0)$  by $\pi$ meet $E_f$  and $E_0$ respectively at  smooth points of $E$. With this, we have the following result as direct consequence of Lemma \ref{lipshitz}
\begin{coro}\label{bilipshitz}
Given any element $D$ of $\Omega_{f}^{\pi}$ and any neighborhood $U$ of $\Pi_{D}^* \cap \pi^{-1}(B_{\epsilon} \cap X)$  in $X_{\pi} \cap \pi^{-1}(B_{\epsilon} \cap X)$, the local bilipschitz constant $K_D$ of the morphism $\Phi_{D}$ is bounded on $B_{\epsilon} \cap (X \backslash \pi(U))$.
\end{coro}
\qed

Now, we can define the notion of generic linear form and generic polar curve  with respect to $f$ and $\pi$. The following  definition is an adaptation of the notion of generic projection defined in \cite[Subsection 2.2]{BFP}.
  \begin{defi}\label{genericnashrelative}
We say that a linear form $\ell:(X,0) \longrightarrow (\mathbb{C},0)$ is \textbf{generic} with respect to $f$ and $\pi$ if it is the dual linear form of an element $D$ of $\Omega_f^\pi $. In which case We say that $\Pi_D$ is a generic polar curve.\end{defi}


\subsection{Inner rates of a Milnor fibration}
We are now ready to state and prove the main theorem of this section which allows us to see the inner rates as metric invariants of the Milnor fibration of the function $f$. It is also a relative version of \cite[Lemma 3.2]{BFP}.
\begin{thm}\label{Milnor} Let $(X,0) \subset (\mathbb{C}^n,0)$ be a complex surface germ with an isolated singularity  and  let $f:(X,0) \longrightarrow (\mathbb{C},0)$  be a non constant holomorphic function. Let $ \pi: (X_{\pi},E) \longrightarrow (X,0)$ be a good resolution which factors through the Nash transform of $X$ relative to $f$ and the blowup of the maximal ideal of $(X,0)$.

 There exists a rational number $q_v^f \in \mathbb{Q}_{>0}$ such that  for any pair of curvettes   $\gamma_1^*$ and $\gamma_2^*$  of an irreducible component $E_v$ of the exceptional divisor $E$ verifying $\gamma_i^*\cap f^*=\emptyset$ for $i \in \{1,2\}$ we have
$$ \mathrm{d}_{\epsilon}(\gamma_1 \cap f^{-1}(\epsilon),\gamma_2 \cap f^{-1}(\epsilon)) = \Theta(\epsilon^{q_v^f}),$$where $\gamma_1=\pi(\gamma_1^*)$,$\gamma_2=\pi(\gamma_2^*)$ and $\mathrm{d}_{\epsilon}$ is the Riemanian metric induced by $\mathbb{C}^n$ on the Milnor fiber $f^{-1}(\epsilon)$. Furthermore we have  $q_v^f=q_{\ell,v}^f$ whenever $\ell$ is a generic linear form with respect to $f$ and $\pi$.
\end{thm}
\begin{proof}
Let $p$ be a  smooth point of $E$ in $E_v \backslash f^*$. Since $\pi$ factors through the Nash transform of $X$ relative to $f$ and the blowup of the maximal ideal, it has no base point for the family of the generic  polar curves and the  family of the generic hyperplane sections  (by remark \ref{basepoint}). Thus, there exists  a generic linear form $\ell$ such that $p \notin \ell^*  \cup \Pi_{\Phi}^* $ where $\Phi=(\ell,f)$. By proposition \ref{inner-rate} there exists an open neighborhood $O_p \subset E_v$ of $p$ such that for every pair of curvettes $\gamma_{1}^*, \gamma_{2}^*$ of $E_v$ verifying:
\begin{enumerate}
\item $\gamma_1^* \cap \gamma_2^* = \emptyset$
\item $ \gamma_{i}^* \cap O_p \neq  \emptyset$, for $i=1,2$,
\end{enumerate} we have: $$ \mathrm{d}(\Phi(\gamma_1) \cap \{u_2 = \epsilon \},\Phi(\gamma_2) \cap \{u_2= \epsilon \} ) =\Theta(\epsilon^{q_{\ell,v}^f}),$$where  $\gamma_1= \pi(\gamma_1^*), \gamma_2=\pi(\gamma_2^*)$ and $\epsilon \in \mathbb{R}$. By using Corollary \ref{lipshitz} we deduce that 
$$ \mathrm{d}_{\epsilon}(\gamma_1 \cap f^{-1}(\epsilon),\gamma_2 \cap f^{-1}(\epsilon) ) =\Theta(\epsilon^{q_{\ell,v}^f}).$$This means that the number $q_{\ell,v}^f$ does not depend on the choice of the generic linear form $\ell$, it will be then denoted $q_v^f$.

Let us now take any pair of curvettes $\gamma^*$ and $\gamma'^*$ meeting $E_v$ at two distinct smooth points $p$ and $p'$ of $E$. Since $\mathrm{d}_{\epsilon}(\gamma \cap f^{-1}(\epsilon), \tilde{\gamma} \cap f^{-1}(\epsilon) ) =\Theta(\epsilon^{q_{v}^f})$ for any neighbor  curve germ $\tilde{\gamma}$ of $\gamma= \pi(\gamma^*)$  we have 
$$\mathrm{d}_{\epsilon}(\gamma \cap f^{-1}(\epsilon), \gamma' \cap f^{-1}(\epsilon) ) =\Theta(\epsilon^{q_v(\gamma,\gamma')}) \ \text{with} \ q_v(\gamma,\gamma') \leq q_{v}^f,$$
where $\gamma= \pi(\gamma^*)$ and $\gamma'= \pi(\gamma'^*)$.\\ By compactness of $E$, we can choose a finite sequence of smooth points $p=p_1,p_2,\ldots,p_s=p'$  and a  curvette $\gamma_i^*$ passing through each $p_i$ such that:

$$\mathrm{d}_{\epsilon}(\gamma_i \cap f^{-1}(\epsilon), \gamma_{i+1} \cap f^{-1}(\epsilon) ) =\Theta(\epsilon^{q_v^f}) $$ We then have 
\begin{eqnarray*}    K' \epsilon^{q_v(\gamma,\gamma')}   \geq  \mathrm{d}_{\epsilon}(\gamma \cap f^{-1}(\epsilon), \gamma' \cap f^{-1}(\epsilon) ) \geq \sum_{i=1}^{s}\mathrm{d}_{\epsilon}(\gamma_i \cap f^{-1}(\epsilon), \gamma_{i+1} \cap f^{-1}(\epsilon) ) \geq K \epsilon^{q_v^f}. \end{eqnarray*} where $K$ and $K'$ are constants. We deduce that $q_v(\gamma,\gamma') \geq q_v^f$, which ends the proof.

 \end{proof}

\section{Application to Lipchitz-Killing curvature}
As an application of Theorem \ref{Milnor} we will prove theorem \ref{thmE} which generalizes the result of Garc\'{\i}a Barosso and Teissier in \cite{GarciaBarrosoTeissier1999}.

\subsection{Lipchitz-Killing curvature}
Let $X$ be a real smooth submanifold of $\mathbb{R}^n$ of dimension $m$. Let $T_pX$ be the tangent space of $X$ at a point $p$. Let $\overrightarrow{n_p}$ be a normal vector of $T_pX$ in $\mathbb{R}^n$. Consider $V_{\overrightarrow{n_p}}$ the vector space generated by $T_pX$ and $\overrightarrow{n_p}$. Let $P_{\overrightarrow{n_p}}$  be the orthogonal projection on $V_{\overrightarrow{n_p}}$ and   $X_{\overrightarrow{n_p}}=P_{\overrightarrow{n_p}}(X)$.
\begin{defi}\cite{langevincourbure}\label{langevincourbure}
The \textbf{Lipchitz-Killing curvature} of $X$ at the point $p$ is:

$$ K_{X}(p)=\frac{\omega_{m} }{2\omega_{n-1}}\int_{\overrightarrow{n_p} \in \nu(p) }^{} k_{X_{\overrightarrow{n_p}}} (P_{ \overrightarrow{n_p}} (p)),$$
where $\omega_{i}$ is the volume of the unitary sphere $\mathbb{S}^{i}$ of $\mathbb{R}^{i+1}$, $\nu(p)$ is the set of normal vectors at $p$ and ${k_{X_{\overrightarrow{n_p}} }}$ is the classical Gaussian curvature.
\end{defi}
Now let us state the \textbf{exchange formula} due to Rémi Langevin which provides an easy way to compute the integral of the Lipshitz-Killing curvature in the case of a compact complex submanifold of $\mathbb{C}^n$ possibly with boundaries .

\begin{thm}\cite[Theorem A.III.3']{langevinthese}\label{langevinshifrin}
Let $X$ be a compact complex submanifold embedded in  $\mathbb{C}^n$ of dimension $m$ then

$$ (-1)^m \int_{p \in M} K_X(p) \mathrm{d}V= \frac{\pi \omega_{2m} }{2\omega_{2n-1}}\int_{H \in \mathbb{G}^{n-1}(\mathbb{C}^n)} \mathrm{Card}\{z \in X \ | \ T_zX \subset H\} \mathrm{d}H, $$where $\mathrm{d}V$ and $\mathrm{d}H$ are respectively the volume forms of $X$ and the grassmannian of hyperplanes  $ \mathbb{G}^{n-1}(\mathbb{C}^n)$.
\end{thm}
\subsection{The Lipshitz-Killing curvature of a Milnor fiber}

Let $(X,0)$ be a complex surface germ with an isolated singularity embedded in $\mathbb{C}^n$ and $f:(X,0) \longrightarrow (\mathbb{C},0)$ be a non constant holomorphic function. For any element  $H$ of    $\mathbb{G}^{n-1}(\mathbb{C}^n)$ denote  by $\ell_H$ a linear form whose kernel is  $H$. Consider the morphism $\Phi_H=(\ell_H,f)$ and denote by $\Pi_{H}$ the associated polar curve. Let us now apply Theorem \ref{langevinshifrin} on the Milnor fiber $F_{\epsilon,t}:=f^{-1}(t) \cap B_{\epsilon}$:
$$  \int_{p \in F_{\epsilon,t} } K_{F_{\epsilon,t}}(p) \mathrm{d}V= \frac{-\pi \omega_{2} }{2\omega_{2n-1}}\int_{H \in \mathbb{G}^{n-1}(\mathbb{C}^n)} \mathrm{Card}\{ \Pi_{H} \cap F_{\epsilon,t} \} \mathrm{d}H. $$On the other hand, for every element $H$ of  $\mathbb{G}^{n-1}(\mathbb{C}^n)$  there exists a positive real number $\delta_{\epsilon}^{H} $ such that for every $t $ in $\mathbb{C}$ whose module is less than  or equal to $\delta_{\epsilon}^{H}$ we have  $$ \mathrm{Card}\{ F_{\epsilon,t} \cap \Pi_H \} = \Pi_H \cdot f^{-1}(0).$$Let $ \delta_{\epsilon} := \mathrm{inf}\{\delta_{\epsilon}^{H} \ | \ H \in  \mathbb{G}^{n-1}(\mathbb{C}^n) \}$. Thus, by taking  the limit we get
\begin{equation}\label{langevin}\lim\limits_{\epsilon \rightarrow 0, |t| <{\delta_{\epsilon}}}  \int_{p \in F_{\epsilon,t} } K_{F_{\epsilon,t}}(p) \mathrm{d}V= \frac{-\pi \omega_{2} }{2\omega_{2n-1}} \int_{H \in \mathbb{G}^{n-1}(\mathbb{C}^n)} \Pi_H \cdot f^{-1}(0)\mathrm{d}H.\end{equation} 

 As a direct consequence of (\ref{langevin}) and Theorem \ref{francoise} we obtain the following  result.

\begin{thm}\label{langevin+michel}
Let $(X,0)$ be a complex surface germ with an isolated singularity embedded in $(\mathbb{C}^n,0)$ and let $f:(X,0) \longrightarrow (\mathbb{C},0)$ be a non constant holomorphic function. Let $\pi:(X_{\pi},E) \longrightarrow (X,0)$ be a good resolution of $(X,0)$, then:$$\lim\limits_{\epsilon \rightarrow 0, |t| <{\delta_{\epsilon}}} \int_{p \in F_{\epsilon,t} } K_{F_{\epsilon,t}}(p) \mathrm{d}V=\frac{\pi \omega_{2} }{2\omega_{2n-1}} \mathrm{Vol}(\mathbb{G}^{n-1}(\mathbb{C}^n))\sum_{v \in V(\Gamma_{\pi})} m_{v}(f)\chi'_{v},$$
where $\chi'_v:=2-2g_v-\mathrm{val}(v) -f^* \cdot E_v-\ell_H^* \cdot E_v.$

\end{thm}\qed
\begin{rema}Theorem \ref{langevin+michel} was proved in \cite{langevinformula} when $(X,0)=(\mathbb{C}^k,0)$.
\end{rema}

Let $\pi:(X_{\pi},E) \longrightarrow (X,0)$ be a good resolution of $(X,0)$ and let $E_v$ be an irreducible component of $E$. Let $\mathcal{N}(E_v,\epsilon), \epsilon >0$ be a family of tubular neighborhoods of $E_v$ in $X_{\pi}$ such that $$\lim\limits_{\epsilon \rightarrow 0}\mathcal{N}(E_v,\epsilon)=E_v, $$
and such that  the  set $\mathrm{Horn}(\epsilon,v):= \pi(\mathcal{N}(E_v,\epsilon) )$ is included in $B_{\epsilon}$. For example one can choose a Riemanian metric on $X_{\pi}$ and consider the set of points which are close enough to $E_v$ with respect to that metric. Consider  now the intersection of the Milnor fiber of $f$ with this set
$$ F_{\epsilon,t}^v=f^{-1}(t) \cap \mathrm{Horn}(\epsilon,v). $$ 
Let  $\delta_{\epsilon}^{H}>0$   be such that for any complex number $t$ whose module is less then or equal to $\delta_{\epsilon}^{H}$  we have $$ \mathrm{Card}\{ F_{\epsilon,t}^v \cap \Pi_H \} = m_v(f) \Pi_H^* \cdot E_v,$$
where $\Pi_{H}^*$ is the strict transform of $\Pi_{H}$ by $\pi$. 

\begin{thm}\label{concentration}Let $(X,0)$ be a complex surface germ with an isolated singularity embedded in $(\mathbb{C}^n,0)$ and let $f:(X,0) \longrightarrow (\mathbb{C},0)$ be a non constant holomorphic function. Let $\pi:(X_{\pi},E) \longrightarrow (X,0)$ be a good resolution of $(X,0)$ which factors through the Nash transform of $X$ relative to $f$ and the blowup of the maximal ideal of $(X,0)$. Let $v$ be a vertex of $\Gamma_\pi$, then:

$$\lim\limits_{\epsilon \rightarrow 0, |t| <{\delta_{\epsilon}}}  \int_{p \in F_{\epsilon,t}^v} K_{F_{\epsilon,t}^v}(p) \mathrm{d}V=\frac{\pi \omega_{2} }{2\omega_{2n-1}} \mathrm{Vol}(\mathbb{G}^{n-1}(\mathbb{C}^n)) \mathrm{C}_f , $$
where   $$C_f=m_v(f)\left( 2g_v-2+ \mathrm{Val}_{\Gamma_{\pi}}(v)+f^* \cdot E_v-\sum_{i \in V(\Gamma_{\pi}) }m_i(f)q_{i}^fE_i \cdot E_v \right).$$
\end{thm}

\begin{proof}
It is a direct consequence of Theorems \ref{laplacien}, \ref{Milnor} and \ref{langevinshifrin}.
\end{proof}

\begin{rema}
Garc\'{\i}a Barosso and Teissier gave in  \cite{GarciaBarrosoTeissier1999} an analogous result to Theorem \ref{concentration}  when $f$ is a germ of non constant holomorphic function at the origin of $\mathbb{C}^2$.

\end{rema}

\bibliographystyle{alpha}                              
\bibliography{biblio}

\begin{thebibliography}{BdSFNP22}

\bibitem[Art66]{artin}
Michael Artin.
\newblock On isolated rational singularities of surfaces.
\newblock {\em Amer. J. Math.}, 88:129--136, 1966.

\bibitem[BdSFNP22]{BFNP}
Andr\'{e} Belotto~da Silva, Lorenzo Fantini, Andr\'{a}s N\'{e}methi, and Anne
  Pichon.
\newblock Polar exploration of complex surface germs.
\newblock {\em Trans. Amer. Math. Soc.}, 375(9):6747--6767, 2022.

\bibitem[BdSFP22a]{BFP}
Andr\'{e} Belotto~da Silva, Lorenzo Fantini, and Anne Pichon.
\newblock Inner geometry of complex surfaces: a valuative approach.
\newblock {\em Geom. Topol.}, 26(1):163--219, 2022.

\bibitem[BdSFP22b]{BFP2}
Andr\'{e} Belotto~da Silva, Lorenzo Fantini, and Anne Pichon.
\newblock On {L}ipschitz normally embedded complex surface germs.
\newblock {\em Compos. Math.}, 158(3):623--653, 2022.

\bibitem[BNP14]{BNP}
Lev Birbrair, Walter~D. Neumann, and Anne Pichon.
\newblock The thick-thin decomposition and the bilipschitz classification of
  normal surface singularities.
\newblock {\em Acta Math.}, 212(2):199--256, 2014.

\bibitem[CA93]{casas}
E.~Casas-Alvero.
\newblock Singularities of polar curves.
\newblock {\em Compositio Math.}, 89(3):339--359, 1993.

\bibitem[DM03]{DM2003}
F.~Delgado and H.~Maugendre.
\newblock Special fibres and critical locus for a pencil of plane curve
  singularities.
\newblock {\em Compositio Math.}, 136(1):69--87, 2003.

\bibitem[DM21]{DM2021}
F.~Delgado and H.~Maugendre.
\newblock Pencils and critical loci on normal surfaces.
\newblock {\em Rev. Mat. Complut.}, 34(3):691--714, 2021.

\bibitem[GB00]{EGB}
Evelia~R. Garc\'{\i}a~Barroso.
\newblock Sur les courbes polaires d'une courbe plane r\'{e}duite.
\newblock {\em Proc. London Math. Soc. (3)}, 81(1):1--28, 2000.

\bibitem[GBT99a]{EGBT}
Evelia Garc\'{\i}a~Barroso and Bernard Teissier.
\newblock Concentration multi-\'{e}chelles de courbure dans des fibres de
  {M}ilnor.
\newblock {\em Comment. Math. Helv.}, 74(3):398--418, 1999.

\bibitem[GBT99b]{GarciaBarrosoTeissier1999}
Evelia Garc\'{\i}a~Barroso and Bernard Teissier.
\newblock Concentration multi-\'{e}chelles de courbure dans des fibres de
  {M}ilnor.
\newblock {\em Comment. Math. Helv.}, 74(3):398--418, 1999.

\bibitem[GR21]{GignacRuggiero2017}
William Gignac and Matteo Ruggiero.
\newblock Local dynamics of non-invertible maps near normal surface
  singularities.
\newblock {\em Mem. Amer. Math. Soc.}, 272(1337):xi+100, 2021.

\bibitem[Gra62]{Greu}
Hans Grauert.
\newblock \"{U}ber {M}odifikationen und exzeptionelle analytische {M}engen.
\newblock {\em Math. Ann.}, 146:331--368, 1962.

\bibitem[HP03]{PH}
Jean-Pierre Henry and Adam Parusi\'{n}ski.
\newblock Existence of moduli for bi-{L}ipschitz equivalence of analytic
  functions.
\newblock {\em Compositio Math.}, 136(2):217--235, 2003.

\bibitem[HW08]{hardy}
G.~H. Hardy and E.~M. Wright.
\newblock {\em An introduction to the theory of numbers}.
\newblock Oxford University Press, Oxford, sixth edition, 2008.
\newblock Revised by D. R. Heath-Brown and J. H. Silverman, With a foreword by
  Andrew Wiles.

\bibitem[Jon15]{Jonsson2015}
Mattias Jonsson.
\newblock Dynamics of {B}erkovich spaces in low dimensions.
\newblock In {\em Berkovich spaces and applications}, volume 2119 of {\em
  Lecture Notes in Math.}, pages 205--366. Springer, Cham, 2015.

\bibitem[KP04]{KTP}
Tzee-Char Kuo and Adam Parusi\'{n}ski.
\newblock Newton-{P}uiseux roots of {J}acobian determinants.
\newblock {\em J. Algebraic Geom.}, 13(3):579--601, 2004.

\bibitem[Lan79]{langevinformula}
R\'{e}mi Langevin.
\newblock Courbure et singularit\'{e}s complexes.
\newblock {\em Comment. Math. Helv.}, 54(1):6--16, 1979.

\bibitem[Lan80a]{langevincourbure}
R.~Langevin.
\newblock Singularit\'{e}s complexes, points critiques et int\'{e}grales de
  courbure.
\newblock In {\em S\'{e}minaire {P}ierre {L}elong-{H}enri {S}koda ({A}nalyse).
  {A}nn\'{e}es 1978/79 ({F}rench)}, volume 822 of {\em Lecture Notes in Math.},
  pages 129--143. Springer, Berlin, 1980.

\bibitem[Lan80b]{langevinthese}
R\'{e}mi Langevin.
\newblock {\em Courbures, feuilletages et surfaces}, volume~3 of {\em
  Publications Math\'{e}matiques d'Orsay 80 [Mathematical Publications of Orsay
  80]}.
\newblock Universit\'{e} de Paris-Sud, D\'{e}partement de Math\'{e}matiques,
  Orsay, 1980.
\newblock Mesures et distributions de Gauss. [Gaussian measures and
  distributions], Dissertation, Universit\'{e} Paris-Sud, Orsay, 1980, With an
  English summary.

\bibitem[Lau71]{laufer1972}
Henry~B. Laufer.
\newblock {\em Normal two-dimensional singularities}.
\newblock Princeton University Press, Princeton, N.J.; University of Tokyo
  Press, Tokyo, 1971.
\newblock Annals of Mathematics Studies, No. 71.

\bibitem[Mer77]{merle}
M.~Merle.
\newblock Invariants polaires des courbes planes.
\newblock {\em Invent. Math.}, 41(2):103--111, 1977.

\bibitem[Mic08]{Michel2008}
Fran\c{c}oise Michel.
\newblock Jacobian curves for normal complex surfaces.
\newblock In {\em Singularities {II}}, volume 475 of {\em Contemp. Math.},
  pages 135--150. Amer. Math. Soc., Providence, RI, 2008.

\bibitem[MM20]{MaugendreMichel2017}
H{\'e}l{\`e}ne Maugendre and Fran{\c{c}}oise Michel.
\newblock On the growth behaviour of {H}ironaka quotients.
\newblock {\em J. Singul.}, 20:31--53, 2020.

\bibitem[Mum61]{mum}
David Mumford.
\newblock The topology of normal singularities of an algebraic surface and a
  criterion for simplicity.
\newblock {\em Inst. Hautes \'{E}tudes Sci. Publ. Math.}, (9):5--22, 1961.

\bibitem[N\'99]{nem}
A.~N\'{e}methi.
\newblock Five lectures on normal surface singularities.
\newblock In {\em Low dimensional topology ({E}ger, 1996/{B}udapest, 1998)},
  volume~8 of {\em Bolyai Soc. Math. Stud.}, pages 269--351. J\'{a}nos Bolyai
  Math. Soc., Budapest, 1999.
\newblock With the assistance of \'{A}gnes Szil\'{a}rd and S\'{a}ndor
  Kov\'{a}cs.

\bibitem[Neu81]{neu81}
Walter~D. Neumann.
\newblock A calculus for plumbing applied to the topology of complex surface
  singularities and degenerating complex curves.
\newblock {\em Trans. Amer. Math. Soc.}, 268(2):299--344, 1981.

\bibitem[NP07]{NP2007}
Walter~D. Neumann and Anne Pichon.
\newblock Complex analytic realization of links.
\newblock In {\em Intelligence of low dimensional topology 2006}, volume~40 of
  {\em Ser. Knots Everything}, pages 231--238. World Sci. Publ., Hackensack,
  NJ, 2007.

\bibitem[Sha13]{shafarevich}
Igor~R. Shafarevich.
\newblock {\em Basic algebraic geometry. 1}.
\newblock Springer, Heidelberg, third edition, 2013.
\newblock Varieties in projective space.

\bibitem[Spi90]{Spivakovsky1990}
Mark Spivakovsky.
\newblock Sandwiched singularities and desingularization of surfaces by
  normalized {N}ash transformations.
\newblock {\em Ann. of Math. (2)}, 131(3):411--491, 1990.

\bibitem[Tei82]{Teissier1982}
Bernard Teissier.
\newblock Vari\'{e}t\'{e}s polaires. {II}. {M}ultiplicit\'{e}s polaires,
  sections planes, et conditions de {W}hitney.
\newblock In {\em Algebraic geometry ({L}a {R}\'{a}bida, 1981)}, volume 961 of
  {\em Lecture Notes in Math.}, pages 314--491. Springer, Berlin, 1982.

\end{thebibliography}

\vfill

\end{document}